\documentclass[letter,11pt]{article}

\usepackage{authblk}
\usepackage{amsmath}
\usepackage{amssymb}
\usepackage{amsthm}
\usepackage{bm}
\usepackage{enumitem}
\usepackage{fullpage}
\usepackage{hyperref}
\hypersetup{
citecolor=black,
filecolor=black,
linkcolor=red,
urlcolor=black,
linktoc=all
}
\usepackage{lmodern}
\usepackage{mathalpha}
\usepackage{mathtools}
\usepackage{microtype}
\usepackage{accents}
\usepackage{booktabs}
\usepackage{array}
\usepackage{xcolor,colortbl}

\usepackage{tikz}
\usetikzlibrary{arrows}

\usepackage[capitalise]{cleveref}
	
\newcommand{\ie}{\textit{i.e.}}
\newcommand{\eg}{\textit{e.g.}}

\newcommand{\etc}{\textit{etc}}

\newcommand{\eps}{\varepsilon}
\newcommand{\thet}{\vartheta}

\newcommand{\R}{\mathbb{R}}

\newcommand{\T}{\mathbb{T}}

\newcommand{\Ccal}{\mathcal{C}}
\newcommand{\Kcal}{\mathcal{K}}
\newcommand{\Mcal}{\mathcal{M}}
\newcommand{\Ncal}{\mathcal{N}}
\newcommand{\Pcal}{\mathcal{P}}

\newcommand{\Tcal}{\mathcal{T}}

\newcommand{\tplus}{\vee}

\newcommand{\onevector}{\text{e}}
\DeclareMathOperator{\tsegm}{\mathrm{tsegm}}
\DeclareMathOperator{\tbar}{\mathrm{tbary}}

\newcommand{\distInf}{{d_{\infty}}}

\newcommand{\Pt}{\Pcal_t}
\newcommand{\Ptrop}{\mathsf{P}}
\newcommand{\K}{\Kcal}
\newcommand{\Kt}{\K_t}
\newcommand{\Ktrop}{\mathsf{K}}
\newcommand{\ctrop}{\mathsf{c}}
\newcommand{\cp}{\mathcal{C}}
\newcommand{\cptrop}{\mathsf{C}}
\newcommand{\cpt}{\mathcal{C}_t}

\newcommand{\Minimize}{\textrm{Minimize}}
\newcommand{\SubjectTo}{\textrm{subject to}}
\DeclareMathOperator{\linr}{linr}

\newcommand{\e}{\text{e}}
\newcommand{\transpose}[1]{#1^\top}
\newcommand{\ubar}[1]{\underaccent{\bar}{#1}}
\newcommand{\cex}{\textsc{cex}}
\newcommand{\tcex}{\textsc{tcex}}
\newcommand{\optval}{\mathrm{val}}
\newcommand{\multlow}{{\ubar m}}
\newcommand{\multupp}{{\bar m}}
\newcommand{\Mpc}{M^{\mathrm{pc}}}

\DeclareMathOperator{\Diag}{\mathrm{Diag}}
\DeclareMathOperator{\val}{\mathrm{val}}

\DeclareMathOperator{\supp}{\mathrm{supp}}

\DeclareMathOperator{\interior}{\mathrm{int}}

\DeclarePairedDelimiterX{\inp}[2]{\langle}{\rangle}{#1, #2}
\DeclarePairedDelimiterX{\abs}[1]{\lvert}{\rvert}{#1}
\DeclarePairedDelimiterX{\norm}[1]{\lVert}{\rVert}{#1}
\newcommand{\inpT}[3][]{\inp[#1]{#2}{#3}_\T}

\renewcommand{\geq}{\geqslant}
\renewcommand{\leq}{\leqslant}

\newcommand{\unitAM}{e}
\newtheorem{theorem}{Theorem}
\newtheorem{proposition}[theorem]{Proposition}
\newtheorem{lemma}[theorem]{Lemma}
\newtheorem{corollary}[theorem]{Corollary}
\newtheorem{assumption}[theorem]{Assumption}

\theoremstyle{remark}
\newtheorem{remark}[theorem]{Remark}

\title{No self-concordant barrier interior point method\\
is strongly polynomial}

\author[1,2]{Xavier Allamigeon}
\author[1,2]{Stéphane Gaubert}
\author[2,1]{Nicolas Vandame}
\affil[1]{Inria, France}
\affil[2]{CMAP, CNRS, Ecole Polytechnique, IP Paris, France}
\date{November 4, 2021}

\begin{document}

\maketitle

\begin{abstract}
It is an open question to determine if the theory of self-concordant barriers can provide an interior point method with strongly polynomial complexity in linear programming. In the special case of the logarithmic barrier, it was shown in [Allamigeon, Benchimol, Gaubert and Joswig, SIAM J.~on Applied Algebra and Geometry, 2018] that the answer is negative. In this paper, we show that \emph{none} of the self-concordant barrier interior point methods is strongly polynomial. This result is obtained by establishing that, on parametric families of convex optimization problems, the log-limit of the central path degenerates to a piecewise linear curve, independently of the choice of the barrier function. We provide an explicit linear program that falls in the same class as the Klee--Minty counterexample, \ie, in dimension $n$ with $2n$~constraints, in which the number of iterations is $\Omega(2^n)$.  
\end{abstract}

\section{Introduction}

The theory of self-concordant barriers, introduced by Nesterov and Nemirovskii~\cite{NesterovN94} in the nineties, is the cornerstone of modern interior point methods (IPM) for convex optimization. It is the basis of some of the most recent breakthroughs in the complexity of linear programming and combinatorial optimization,
like~\cite{LeeSidford14}. 
Given a convex optimization program of the form
\begin{equation}
\Minimize \enspace \inp{c}{x} \enspace \SubjectTo \enspace x \in \Kcal \, , \label{eq:convex}
\end{equation}
where $\Kcal \subset \R^n$ is a convex body and $c \in \R^n$, interior point methods based on self-concordant barriers consist in solving the following penalized problems:
\begin{equation}
\Minimize \enspace \eta \inp{c}{x} + f(x) \enspace \SubjectTo \enspace x \in \interior \Kcal \label{eq:penalized}
\end{equation}
where $f$ is a self-concordant barrier over the interior $\interior \Kcal$ of $\Kcal$, and $\eta \geq 0$ is a parameter. In more details, the function $f$ is strictly convex, and tends to $+\infty$ when $x$ goes to the boundary of $\Kcal$. Thus, every problem of the form~\eqref{eq:penalized} has a unique optimal solution $\cp(\eta)$, and the function $\eta \mapsto \cp(\eta)$ defines a parametric curve called the \emph{central path}. This curve converges to an optimal solution of the original problem~\eqref{eq:convex} when $\eta \to +\infty$. 
Self-concordant barrier interior point methods consists in approximately following the central path up to sufficiently large value of $\eta$. In a nutshell, starting from an approximation of the point $\cp(\eta)$, they perform a certain number of Newton iterations in order to get an approximation of $\cp(\eta')$, where the ratio $\eta / \eta'$ is bounded by a constant less than $1$. One of the most important results in the theory of self-concordant barriers is that the total number of iterations performed to get an $\eps$-approximation\footnote{\ie, a point $x \in \K$ such that $\inp{c}{x} \leq \mathrm{val} + \eps$, where $\mathrm{val}$ is the optimal value.} of the optimal value of~\eqref{eq:convex} is in $O(\thet_f \log (1 / \eps))$, where $\thet_f$ is the so-called \emph{complexity value} of the barrier $f$, see~\cite[\S~2.3.1]{Renegar01}. In the case of linear programming, \ie, when $\Kcal$ is a polyhedron, we seek for an exact solution of~\eqref{eq:convex}. In this case, the precision $\eps$ is set to $O(1/2^{2L})$, where $L$ is the bitsize of the input (\ie, the total bitsize of the numerical entries of the linear program), after which a ``rounding'' method can be applied to find an exact solution in the boundary of the polyhedron, starting from the $\eps$-approximation point that lies in the interior. Then, the complexity bound reduces to $O(\thet_f \log L)$.

Improving the complexity value $\thet_f$ of the barrier as well as the complexity of every Newton iteration has received a considerable attention over the last few years (see~\cite{BubeckEldan15, LeeSidford14, LeeYue18, LeeSidford19,Brand} for a nonexhaustive list). While this has led to significant improvements in the complexity of linear programming, these contributions cover only one part of the complexity bound of interior point methods. Making any substantial progress on the understanding of the number of iterations performed for general self-concordant barriers is a notorious open problem. The main motivation is that the number of iterations is the key factor to possibly get an algorithm that solves linear programming within strongly polynomial complexity. Recall that a \emph{strongly polynomial algorithm} is a polynomial time algorithm that uses a number of arithmetic operations bounded only by a polynomial in the \emph{number of numerical inputs} (rather than the bitsize of these inputs). Finding a strongly polynomial algorithm for linear programming is known as Smale's ninth problem~\cite{Smale98}, and is one of the major open problems in computational optimization. 

\paragraph{Contributions.} In this paper, we show that no interior point method based on self-concordant barrier has strongly polynomial complexity. In more details, we consider the following linear optimization problem over a $n$-dimensional polytope defined by $2n$ inequalities
\begin{equation}\label{eq:cex} 
\allowdisplaybreaks
\begin{array}{rl}
\Minimize & x_n \\
\SubjectTo &
\begin{aligned}[t]
t^{-u_i} \Big(\sum_{j = 1}^{i-1} x_j\Big) + t^{-u_{i+1} + 1} x_i & \leq \sum_{j = i+1}^{n-1} t^{-u_j} x_j + x_n + t^{-u_n} && \text{for all}\enspace i = 1, \dots, n-1\\
\sum_{j = 1}^n x_j & \leq 1 \\
0 \leq x_1 \leq x_2 & \leq \dots \leq x_{n-1} \leq t^{u_n} x_n \, ,
\end{aligned} 
\end{array} \tag{$\cex_n(t)$}
\end{equation}
where $u_k \coloneqq 3 \cdot 2^{k-2} - 1$ for all $k \geq 1$, and $t > 1$ is a real parameter. Note that, in the case of~\eqref{eq:cex}, the input bitsize $L$ is $\Omega(u_n \log_2 t)$.

This is the main result of the paper. 
\begin{theorem}\label{th:cex_intro}
When $t \gg 1$, every self-concordant barrier interior point method requires at least $2^n - 1$ iterations to reduce the value of the objective function from $\Omega(1)$ to $1/2^{2 L}$ in the linear program~\eqref{eq:cex}.
\end{theorem}
The detailed form of this result is \Cref{th:cex_lower_bound}. It requires technical definitions on the neighborhood of the central path used by IPMs.

We provide in Figure~\ref{fig:cube} an illustration of the feasible set of~\eqref{eq:cex} when $n = 3$ with $t = 2.5$. As we can already see there, the vertices obey to rather different asymptotic behaviors when the parameter $t$ grows: some of them get ``very close'', while some others get ``very far''. Indeed, the entries of the vertices are given by rational functions of $t$, and each of them has its own asymptotics $\Theta(t^\alpha)$ ($\alpha \in \R \cup \{-\infty\}$) when $t \gg 1$. In such a case, a common approach that originates from tropical geometry is to apply a logarithmic rescaling, where the base of the logarithm is set to $t$. This leads to introduce the map $\log_t (\cdot) = \frac{\log(\cdot)}{\log t}$, and to examine the image under $\log_t$ of objects such as the feasible set of~\eqref{eq:cex}.\footnote{Here, the map $\log_t$ is understood entrywise: given a point $x \in \R_{\geq 0}^n$, its image under $\log_t$ stands for the point $(\log_t x_i)_i$. We use the convention $\log_t 0 = -\infty$.} This process is loosely referred to as the \emph{tropicalization}.  Results in tropical convexity entail that the image under $\log_t$ of the feasible set of~\eqref{eq:cex} has a limit, and that this limit is a ``tropical polytope.'' 
We point out that no prior knowledge of tropical geometry is necessary to understand the scheme of our approach, and the elementary tools and notations required are motivated and introduced when needed. 

\begin{figure}
\begin{center}
\begin{minipage}{.45\textwidth}
\begin{center}
\begin{tikzpicture}[x  = {(0.9cm,-0.076cm)},
                    y  = {(-0.06cm,0.95cm)},
                    z  = {(-0.44cm,-0.29cm)},
                    scale = 4.5,
                    color = {lightgray},>=stealth']

\begin{scope}
\draw[gray!80!black,->] (0,0,0) -- (0.7,0,0) node[below] {$x_1$};
\draw[gray!80!black,->] (0,0,0) -- (0,1.2,0) node[left] {$x_2$};
\draw[gray!80!black,->] (0,0,0) -- (0,0,1.2) node[left] {$x_3$};	
\end{scope}

  % POINTS STYLE
  \definecolor{pointcolor_p}{rgb}{ 1,0,0 }
  \tikzstyle{pointstyle_p} = [fill=orange]

  % DEF POINTS
  \coordinate (v0_p) at (0.406338, 0.525437, 0.0682252);
  \coordinate (v1_p) at (0.451, 0.451, 0.098);
  \coordinate (v2_p) at (0.0445682, 0.0445682, 0.000456379);
  \coordinate (v3_p) at (0.0569325, 0.0736196, 0.000753865);
  \coordinate (v4_p) at (0, 0.666667, 0.00682667);
  \coordinate (v5_p) at (0, 0, 0);
  \coordinate (v6_p) at (0, 0, 1);
  \coordinate (v7_p) at (0, 0.985023, 0.0149766);

  % EDGES STYLE
  \definecolor{edgecolor_p}{rgb}{ 0,0,0 }

  % FACES STYLE
  \definecolor{facetcolor_p}{rgb}{ 0.4667,0.9255,0.6196 }

  \tikzstyle{facestyle_p} = [fill=gray!30, fill opacity=0.8, draw=edgecolor_p, thick, line cap=round, line join=round]

  % FACES and EDGES and POINTS in the right order
  \draw[facestyle_p] (v5_p) -- (v6_p) -- (v7_p) -- (v4_p) -- (v5_p) -- cycle;
  \draw[facestyle_p] (v0_p) -- (v1_p) -- (v2_p) -- (v3_p) -- (v0_p) -- cycle;
  \draw[facestyle_p] (v7_p) -- (v0_p) -- (v3_p) -- (v4_p) -- (v7_p) -- cycle;
  \draw[facestyle_p] (v2_p) -- (v5_p) -- (v4_p) -- (v3_p) -- (v2_p) -- cycle;

  %POINTS
  \fill[pointstyle_p] (v4_p) circle (0.25 pt);
  \fill[pointstyle_p] (v3_p) circle (0.25 pt);

  %FACETS
  \draw[facestyle_p] (v1_p) -- (v6_p) -- (v5_p) -- (v2_p) -- (v1_p) -- cycle;

  %POINTS
  \fill[pointstyle_p] (v5_p) circle (0.25 pt);
  \fill[pointstyle_p] (v2_p) circle (0.25 pt);

  %FACETS
  \draw[facestyle_p] (v7_p) -- (v6_p) -- (v1_p) -- (v0_p) -- (v7_p) -- cycle;

  %POINTS
  \fill[pointstyle_p] (v7_p) circle (0.25 pt);
  \fill[pointstyle_p] (v6_p) circle (0.25 pt);
  \fill[pointstyle_p] (v1_p) circle (0.25 pt);
  \fill[pointstyle_p] (v0_p) circle (0.25 pt);

  %FACETS

\end{tikzpicture}
\end{center}
\caption{The feasible set of $\cex_3(t)$ with $t = 2.5$.}\label{fig:cube}
\end{minipage}
\hfill
\begin{minipage}{.45\textwidth}
\begin{center}
\include{trop_cp}
\end{center}
\caption{The tropical central path (in blue) of the family of linear programs $(\cex_3(t))$.}\label{fig:trop_cp}
\end{minipage}
\end{center}
\end{figure}

Following the same idea, we investigate the complexity of interior point methods in a more general setting, considering
parametric families of convex optimization problems of the form
\begin{equation}\label{eq:convex_t}
\Minimize \enspace \inp{c_t}{x} \enspace \SubjectTo \enspace x \in \Kt \, , \tag{$\textsc{CP}(t)$}
\end{equation}
where $(\Kt)_{t > 1}$ is a parametric family of closed convex sets included in $\R_{\geq 0}^n$, and $c_t \in \R_{\geq 0}^n$ for all $t > 1$. We suppose that these problems are well-posed, in the sense that the ``log-limits'' of the sets $\Kt$
and of the vectors $c_t$ are required to exist, \ie, the limit when $t \to +\infty$ of their image by the map $\log_t$. 
We denote by $\Ktrop$ and $\ctrop$ these limits.  The choice of an arbitrary self-concordant barrier $f_t$ over $\interior \Kt$ gives rise to a central path that we denote by $\cpt$. Then, under mild assumptions, which apply in particular to parametric families of convex polyhedra and more generally convex semialgebraic sets, we prove that the log-limit of these central paths does exist and has an explicit characterization. The latter uses the notion of \emph{tropical barycenter} of a compact set $S \subset \R^n$, which is defined as the supremum (in the sense of the entrywise ordering of $\R^n$) of the elements of $S$.  
\begin{theorem}[see~\Cref{thm:unif_conv_cp}]
\label{thm:unif_conv_cp_intro}
Suppose that~\Cref{assmpt:assmpt_K_c} holds. Then, the family of functions ${(\log_t \cpt)}_t$ uniformly converges to the map $\cptrop$, where $\cptrop(\lambda)$ is the tropical barycenter of the set $\{ x \in \Ktrop \colon \inpT{\ctrop}{x} \leq -\lambda \}$, and $\inpT{x}{y} \coloneqq \max_i (x_i + y_i)$. 
\end{theorem}
We provide in Figure~\ref{fig:trop_cp} an illustration of the log-limits of the feasible set of~\eqref{eq:cex} (in grey) and their central paths (in blue), when $n = 3$.

Theorem~\ref{thm:unif_conv_cp_intro} generalizes
the result of~\cite{ABGJ18} which described
the tropicalization of the central path
in the special case of the logarithmic barrier and of linear programming.
\Cref{thm:unif_conv_cp_intro} now applies to {\em arbitrary} self-concordant
barriers, and also to convex programming problems that are
more general than linear programming.
This is made possible by showing that any self-concordant barrier essentially behaves like the logarithmic barrier over $\R_{\geq 0}^n$, provided that the domain of the barrier is included in $\R_{\geq 0}^n$. This ``log-like'' property, which is of independent interest, is proved in Section~\ref{sec:log_like_barrier}.

We now informally explain how this characterization can lead to lower bounds on the iteration complexity of IPMs like the one of Theorem~\ref{thm:unif_conv_cp_intro}. As shown in~\cite{ABGJ18} and illustrated in Figure~\ref{fig:trop_cp}, the tropical central path is a piecewise linear curve, and the number of its nondifferentiability points roughly controls the number of linear pieces needed to approximate it. The latter quantity may be large because the tropical central path can lie in the boundary of the tropical feasible set. We prove in Section~\ref{sec:complexity} that the trajectory followed by interior point methods, as described by Renegar in his monograph~\cite{Renegar01}, is fully contained in a multiplicative neighborhood of the central path, and that the log-limit of the latter also coincides with the tropical central path $\cptrop$. In other words, provided that $t \gg 1$, the image by $\log_t$ of the trajectory of the IPM is a tight approximation of the tropical central path. This allows to prove that this trajectory has to contain ``many'' segments, or, equivalently, that  IPM must perform many iterations. More accurately, Theorems~\ref{th:complexity} and~\ref{th:complexity2} prove that the number of iterations is bounded from below by the minimal number of tropical segments needed to describe the tropical central path; see Section~\ref{sec:uniform} for a definition of tropical segments. 

Another novelty of the present work lies in the counterexample of linear program for which interior points methods are not strongly polynomial.
The counterexample described in~\cite{ABGJ18} involved a specific linear program with $3r+1$ inequalities in dimension $2r$, leading to a lower
bound for the iteration complexity in $\Omega(2^{n/2})$ in terms of the number $n$ of variables. The present example leads to a lower bound of $\Omega(2^{n})$ iterations (a stronger bound even in the special case of the logarithmic barrier). Interestingly, the numbers of variables and constraints, $(n,2n)$,
are the same as in the example of Klee and Minty~\cite{KleeMinty} or other cubes (see~\eg~\cite{amenta1999deformed}) arising in the study of the worst-case complexity of the simplex method, as well as of the combinatorial diameter of polytopes~\cite{KleeWalkup67}. In fact, the analysis of the tropical central path of~\eqref{eq:cex} bears resemblance with that of the simplex path in  deformed cubes like the one of Klee--Minty: the tropical central path~\eqref{eq:cex} consists of two copies of the tropical central path of~($\cex_{n-1}(t)$) located in the neighborhood of disjoint facets of~\eqref{eq:cex}, separated by an extra tropical segment. In this way, we prove that the tropical central path consists in the concatenation of $2^{n}-1$ tropical segments (Theorem~\ref{th:cex_gamma}). 

\paragraph{Related work.} The theory of self-concordant barriers has been introduced by Nesterov and Nemirovskii, and we refer to~\cite{NesterovN94} and Renegar's monograph~\cite{Renegar01} for a complete account. One notable result of~\cite{NesterovN94} is the definition of the universal barrier for every convex body, with complexity value $O(n)$. This bound has been since refined  into~$n$ by Lee and Yue~\cite{LeeYue18}. Another related result is the work of Bubeck and Eldan~\cite{BubeckEldan15} on the entropic barrier, where they prove that its complexity value is $(1 + o(1))n$. A difficulty with these barriers is that they are defined implicitly, and rather difficult to compute with. To this respect, a remarkable breakthrough has been done by Lee and Sidford~\cite{LeeSidford14,LeeSidford19} (see also~\cite{Fazel2021}) who introduced a polynomial-time computable barrier based on Lewis weights, with complexity value $O(n \log^5 m)$ for $n$-dimensional polytopes defined by $m$ inequalities.

Concerning the study of the iteration complexity of IPMs,  Vavasis and Ye~\cite{VavasisYe} improved the log-barrier predictor-corrector approach~\cite{MizunoToddYe} by taking into account the geometry of the central path. They reduced the number of iterations along ``straight parts'' (this is called the \emph{layered least squares} (LLS) step). This yields a bound $O(\sqrt{n} \log \chi)$ on the number of iterations, where $\chi$ is a condition number of the matrix associated with the constraints. This bound has been later refined by Monteiro and Tsuchiya~\cite{MonteiroTsuchiya} into a condition number that is invariant under diagonal scaling. More recently, Dadush, Huiberts, Natura and Végh~\cite{dadush2020} developed a scaling invariant LLS IPM, and improved the bound of Monteiro and Tsuchiya by exploiting an imbalance measure of the circuits of the matroid associated with the constraints of the linear program. Another iteration complexity measure for the case of the logarithmic barrier, and that takes the form of an integral of a certain ``curvature'' along the central path, has been introduced in the work of Sonnevend, Stoer and Zhao~\cite{sonnevend1991complexity}. 

The link between the complexity of the simplex method and that of (log-barrier) IPMs was initially suggested by Deza, Nematollahi and Terlaky~\cite{dezakleeminty}. They showed that for a Klee--Minty cube with exponentially many redundant constraints, the central path can be forced to visit a neighborhood of each of the $2^n$ vertices of the cube. Linear programs with $(n,2n)$ variables and constraints were also considered by Mut and Terlaky~\cite{mut} in the context of IPMs. They showed that considering these instances is enough to obtain worst case bounds for the Sonnevend's curvature of the central path. 

The idea of applying tropical geometry to the complexity analysis of IPMs was initially developed in the previous work~\cite{ABGJ18} in the case of the logarithmic barrier; see also~\cite{sirev} for a more introductive presentation. A key notion that we use from~\cite{ABGJ18} is the complexity measure of the tropical central path in terms of number of tropical segments. In contrast, the techniques used to study the log-limit of the central path for general barriers have to be completely different. Indeed, the proof of~\cite{ABGJ18} relies in an essential way on the nature of the log-barrier central path that is a piece of a real algebraic curve. This is no longer true for general barriers. We develop here a new approach in which the log-limit of the central path is studied directly from the geometric properties of self-concordant barriers, without any further requirements. We note that the study of the log-limit of the entropic barrier central path appeared in~\cite{entropic}. There again, the proof was tied to the specific form of the barrier (Cramér transform). 

\paragraph{Outline.} The log-like properties of self-concordant barriers are established in~\Cref{sec:log_like_barrier}. In~\Cref{sec:uniform}, we study the log-limit of the central path, and prove~\Cref{thm:unif_conv_cp_intro}. In~\Cref{sec:complexity}, we exploit this result to establish a general lower bound for IPMs. We finally analyze the linear program~\eqref{eq:cex} and its tropical central path in~\Cref{sec:cube}, and prove \Cref{th:cex_intro}. Appendix contains auxiliary results and some of the proofs.

\paragraph{General notation.} We write $[n]$ for the set $\{1, \dots, n\}$. The notation $e$ stands for the all one vector, and for $I \subset [n]$, $\onevector^I$ is the vector whose $i$th component is equal to $1$ if $i \in I$ and $0$ otherwise. Similarly, $e^i$ is the $i$th element of the canonical basis. 

For $n$-vectors $x, y$, we write $x \leq y$ if for all $i \in [n]$, $x_i \leq y_i$, and $x < y$ if for all $i \in [n]$, $x_i < y_i$. We often use the notation $f(x)$ for the vector $(f(x_i))_i$ obtained by applying the function $f$ entrywise.

\section{Log-like properties of self-concordant barriers}
\label{sec:log_like_barrier}

We start by recalling some basic elements of the theory of self concordant barriers, initially developed in~\cite{NesterovN94}. We follow the exposition of~\cite{Renegar01}.

Let $f$ be a real valued $\Ccal^2$ function with domain $D_f$, where $D_f$ is an open convex subset of $\R^n$. We denote its gradient by $g(x)$ and its Hessian by $H(x)$, and we suppose that the latter is positive definite for all $x \in D_f$ (subsequently, $f$ is strictly convex). Every Hessian $H(x)$ ($x \in D_f$) gives rise to an inner product, defined by
\[
\inp{u}{v}_x \coloneqq \inp{u}{H(x)v} \, .
\]
This inner product induces a norm $\norm{\cdot}_x$. We denote by $B_x(y, r) \coloneqq \{ z \in \R^n \colon \norm{y - z}_x < r \}$ the ball with center $y$ and radius $r$ in the sense of $\norm{\cdot}_x$. We denote by $g_x$ and $H_x$ the gradient and Hessian induced by the inner product $\inp{\cdot}{\cdot}_x$, \ie, $g_x(y) \coloneqq H(x)^{-1} g(y)$ and $H_x(y) \coloneqq H(x)^{-1} H(y)$.

The function $f$ is \emph{(strongly nondegenerate) self-concordant} if for all $x \in D_f$ we have $B_x(x,1) \subset D_f$ and for all $y \in B_x(x,1)$ and $v \neq 0$,
\[
1 - \norm{y-x}_x \leq \frac{\norm{v}_y}{\norm{v}_x} \leq \frac{1}{1-\norm{y-x}_x} \, .
\]
The intuition behind this condition is that the change of the norm $\norm{\cdot}_x$ into $\norm{\cdot}_y$ is well controlled when $x$ and $y$ are close to each other. By \cite[Theorem 2.2.9]{Renegar01}, this implies that $f(x) \to +\infty$ and $x \to +\infty$. If $f$ is thrice differentiable, then, for each choice
of $x,d\in \R^n$ the restriction $\phi(t)=f(x+td)$ satisfies $\phi'''(t)\leq 2\phi''(t)^{3/2}$, the property required in the original definition of~\cite{NesterovN94}, see~\cite[\S~2.5]{Renegar01} for a detailed discussion. The function $f$ is furthermore said to be a \emph{$\thet$-self-concordant barrier function} if
\[
\thet \coloneqq \sup_{x \in D_f} \norm{g_x(x)}^2_x < \infty \, .
\]
The quantity $\thet$ is called the \textit{complexity value of $f$}.

The main result of this section is to show that the Hessian of $f$ is well approximated by that of the logarithmic barrier function $x \mapsto \log x$ over the positive orthant $\R_{> 0}^n$, when the latter contains $D_f$. We start by a lemma on the norms $\norm{\cdot}_x$ in this setting, where we denote by $|x|$ the vector with entries $|x_i|$.
\begin{lemma}
\label{lemma:ellipsoid_positive_orthant}
Suppose that $D_f \subset \R_{> 0}^n$. Then for all $x \in D_f$ and $z \in B_x(0;1)$, we have $|z| < x$. In consequence, for any $r>0$ and $y \in B_x(x;r)$, we have $y < (r+1)x$.
\end{lemma}

\begin{proof}
Since $z \in B_x(0;1)$, we have $x+z \in B_x(x; 1)$. As $B_x(x; 1) \subset D_f \subset \R^n_{> 0}$, we deduce $x + z > 0$. As $-z \in B_x(0;1)$, we similarly have $x - z > 0$. Therefore, $|z_i| < x_i$ for all $i \in [n]$. The second statement follows from the first one applied to $z = \frac{1}{r} (y - x)$.
\end{proof}

Given a point $x \in \R_{> 0}^n$, we denote by $1 / x$ (resp.~$1/x^2$) the vector with entries $1 / x_i$ (resp.~$1 / x_i^2$). Besides, if $v$ is a vector, the notation $\Diag(v)$ stands for the diagonal matrix with diagonal entries~$v_i$.
We denote by $\succcurlyeq $ the Loewner order on the space of symmetric matrices,
so that $A\succcurlyeq B$ if $z^\top A z\geq z^\top Bz$ holds for
all $z\in \R^n$. 
The following statement provides a lower bound of the Hessian of $f$:
\begin{proposition}
\label{thm:hessian_ineq_lb}
Suppose that $D_f \subset \R_{> 0}^n$. Then for all $x \in D_f$, we have
\[
H(x) \succcurlyeq \frac{1}{n} \Diag \Big(\frac{1}{x^2}\Big) \, .
\]
\end{proposition}

\begin{proof}
  Let $x \in D_f$. By \cref{lemma:ellipsoid_positive_orthant}, for all $z \in B_x(0; 1)$, we have $|z| < x$, and therefore $\sum_i \big( \frac{z_i}{x_i} \big)^2 < n$.
In other words,
\[
B_x(0; 1) = \big\{ z \in \R^n \colon \transpose{z} H(x) z < 1 \big\} \subset \bigg\{ z \in \R^n \colon \sum_i {\Big( \frac{z_i}{x_i} \Big)}^2 < n \bigg\} \, .
\]
Taking the closure of these two sets, and using the homogeneity in $z$,
we deduce that $\sum_i {\Big( \frac{z_i}{x_i} \Big)}^2 \leq n z^\top H(x)z$
holds for all $z\in \R^n$, meaning that $\frac{1}{n} \Diag\big(\frac{1}{x^2}\big)\preccurlyeq H(x)$.
\end{proof}

Given $x, y \in D_f$, an elementary computation from the lower bound on the second derivative of $s \coloneqq f(x + s(y-x))$ provided by \Cref{thm:hessian_ineq_lb} yields the following statement, whose proof is given in Appendix~\ref{app:proofs}. 
\begin{corollary}\label{cor:log_bound}
Suppose that $D_f \subset \R_{> 0}^n$. Then for all $x,y \in D_f$, we have
\[
f(y) - f(x) \geq \inp[\Big]{g(x) + \frac{1}{nx}}{y-x} + \frac{1}{n} \sum_i \big(\log x_i - \log y_i \big) \, .
\]
\end{corollary}

The next statement provides an upper bound on the Hessian of $f$ at the point $x$, under the condition that this point is well ``inside'' the domain of $f$. We denote $K \coloneqq 4 \thet + 1$.

\begin{theorem}
\label{thm:hessian_ineq_ub}
Suppose that $D_f \subset \R_{> 0}^n$. Consider $x \in D_f$ such that for all $i \in [n]$, $x + (K+1) x_i e_i \in D_f$ and $x-\frac{1}{2} x_i e_i \in D_f$. Then
\[
H(x) \preccurlyeq
4nK^2 \Diag\Big(\frac{1}{x^2}\Big)  \, .
\]
\end{theorem}

\begin{proof}
Consider $x \in D_f$ as in the theorem. Let us show that all the components of $g(x)$ are negative. By contradiction, suppose that there exists $i_0 \in [n]$ such that $g_{i_0}(x) \geq 0$. By~\cite[Theorem~2.3.4]{Renegar01}, we know that all $y \in D_f$ satisfying $\inp{g(x)}{y-x} \geq 0$ belong to $B_x(x; K)$. We apply this result to $y = x + (K+1) x_{i_0} e_{i_0}$. Since $\inp{g(x)}{y-x} = (K+1) g_{i_0}(x) x_{i_0} \geq 0$, then $(K+1) x_{i_0} \onevector_{i_0} \in B_x(0;K)$. By~\cref{lemma:ellipsoid_positive_orthant}, this implies that $(K+1) x_{i_0} \onevector_{i_0} < (K+1) x$, which is a contradiction.

This implies that for all $i \in [n]$, we have $\inp{g(x)}{-\frac{x_i}{2} \onevector_i} \geq 0$ as the inner product of two vectors with negative components. Applying once more \cite[Theorem 2.3.4]{Renegar01}  to $y = x-\frac{x_i}{2}$ gives that $-\frac{x_i}{2} \onevector_i \in B_x(0;K)$. By symmetry, $\frac{x_i}{2} \onevector_i \in B_x(0;K)$ as well. Since the ball $B_x(0; K)$ is convex, we deduce that it contains the convex hull of the points $\pm \frac{x_i}{2} \onevector_i$ for $i \in [n]$, which is precisely the scaled closed $L_1$ ball, $\bar{B}^1_x(0,1/2)
:= \{z \in \R^n|\sum_i |z_i/x_i|\leq 1/2\}$. By the Cauchy-Schwarz inequality,
$\sum_i |z_i/x_i|\leq \sqrt{\sum_i(z_i/x_i)^2}\sqrt{n}$, and so,
$ \sum_i(z_i/x_i)^2 n \leq 1/4 $ implies $z^\top H(x)z \leq K^2$.
By homogeneity, this entails that $z^\top H(x)z \leq 4nK^2 \sum_i(z_i/x_i)^2$, i.e., $H(x)\preccurlyeq 4nK^2 \Diag(\frac{1}{x^2})$.
\end{proof}

\begin{remark}
  The assumption $D_f \subset \R_{> 0}^n$ can be weakened.
  It suffices that $D_f$ be included in a simplicial cone,
  i.e., a cone $C$ of $\R^n$ generated by $n$ linearly independent vectors
  $e_1,\dots,e_n$. Such a cone yields a lattice
  order in $\R^n$, $\leq_C$ defined by $u\leq_C v$ iff $v-u\in C$.
  Equivalently, denoting by $e_1^*,\dots,e_n^*$ the dual canonical basis
  of $e_1,\dots,e_n$, this order is defined by $e_i^*(v)\geq e_i^*(u)$
  for all $i\in [n]$. Then, the notion of absolute value 
  used in~\Cref{lemma:ellipsoid_positive_orthant} is still well defined
  ($|z|:= z\vee -z$ with respect to the lattice order), and the inequalities stated in~\Cref{thm:hessian_ineq_lb} and~\Cref{thm:hessian_ineq_ub} carry over, replacing $\Diag(1/x^2)$
  by the matrix $\sum_i (e_i^*)^\top e_i^*  /(e_i^*(x))^2$, in which $e_i^*$
  is identified to a row vector. The assumption
  that $D_f$ be included in a simplicial cone is satisfied in particular
  if $D_f$ is bounded.
\end{remark}

\section{The log-limit of the central path}\label{sec:uniform}

As explained in the introduction, we consider a parametric family of convex optimization problems of the form~\eqref{eq:convex_t}, where for all $t > 1$, $\Kt$ is a closed convex subset of $\R_{\geq 0}^n$ and $c_t \in \R_{\geq 0}^n$. For the sake of the simplicity, we assume that $0 \in \Kt$, meaning that the optimal value is equal to $0$. 

We suppose that every problem~\eqref{eq:convex_t} comes with a $\thet$-self-concordant barrier $f_t \colon \interior \Kt \to \R$, where $\thet$ is a quantity independent of $t$. For instance, in the case where the sets $\K_t$ are $n$-dimensional convex bodies, we can choose for $f_t$ the universal barrier of $\Kt$~\cite{NesterovN94}, so that $\thet = n$~\cite{LeeYue18}. We note that our assumption also covers the case where the barriers $f_t$ come from different families of barriers (logarithmic, entropic, \etc) with distinct complexity values, since any $\thet$-self-concordant barrier is a $\thet'$-self-concordant barrier for all $\thet' \geq \thet$.

In this setting, every convex optimization program~\eqref{eq:convex_t} gives rise to a central path that we denote by $\cpt(\cdot)$. More precisely, for all $\eta > 0$, the point $\cpt(\eta)$ is the unique optimal solution of
\begin{equation}
\label{eq:param_interior_point_program}
\Minimize \enspace \eta \inp{c_t}{x} + f_t(x) \enspace \SubjectTo \enspace x \in \interior \Kt \, .
\end{equation}
The purpose of this section is to study the log-limit of the family of the central paths $\cpt$. In order to ensure its existence, we make some assumptions on the existence and the properties of the log-limits of the feasible sets $\Kt$ and cost vectors $c_t$. Note that these limits should range over the domain $(\R \cup \{-\infty\})^n$ (recall the convention $\log_t 0 = -\infty$). We set $\T \coloneqq \R \cup \{-\infty\}$ for short. We denote by $d_\infty$ the $\ell_\infty$-metric,\footnote{We point out that $d_\infty$ can be extended to points $u, v \in \T^n$ by setting $\distInf(u,v) \coloneqq \inf \{ \lambda \geq 0 \colon u - \lambda \onevector \leq v \leq u + \lambda \onevector \}$.}
 and, given two closed sets $U, V$, by $d_\infty(U, V)$ the induced Hausdorff distance, \ie, 
\[
\distInf(U,V) \coloneqq \max\Big(\sup_{u \in U} \inf_{v \in V} \distInf(u,v), \sup_{v \in V} \inf_{u \in U} \distInf(u,v)\Big) \, .
\]
\begin{assumption}
There exist $\Ktrop \subset \T^n$ and $\ctrop \in \T^n$ such that:
\label{assmpt:assmpt_K_c}
\begin{enumerate}[label=(\roman*)]
\item\label{item:i} $\distInf(\log_t \Kt, \Ktrop) = O\Big(\frac{1}{\log t}\Big)$ and $\distInf(\log_t c_t, \ctrop) = O\Big(\frac{1}{\log t}\Big)$;
\item\label{item:ii} $\Ktrop$ is a regular set, \ie, it is equal to the closure of its interior;
\item\label{item:iii} $\Ktrop \cap \R^n$ is a semilinear set, \ie, is is a finite union of polyhedra.
\end{enumerate}
\end{assumption}

Let us explain this assumption. As we show below in \Cref{prop:trop_conv}, the log-limit $\Ktrop$ of the convex sets $(\Kt)_{t > 1}$ has the remarkable property of being convex in the tropical sense.
Let us recall that the \emph{tropical (max-plus) semifield} refers to the set $\T$ where the addition is defined as $x \vee y \coloneqq \max(x,y)$, and the multiplication as the usual sum $x+ y$. In this context, the zero and unit elements are respectively given by $-\infty$ and $0$. Then, given $u, v \in \T^n$, the \emph{tropical segment} $\tsegm(u,v)$ between $u$ and $v$ is defined as the set of points of the form $(u + \lambda e) \vee (v + \mu e)$, where $\lambda, \mu \in \T$ satisfies $\lambda \vee \mu = 0$. In more details, the term $u + \lambda e$ is the tropical analogue of the multiplication of the vector $u$ by the scalar $\lambda$. Then, the term $(u + \lambda e) \vee (v + \mu e)$ is analogous to the weighted addition of the vectors $u, v$.\footnote{We extent the addition $\vee$ entrywise, meaning that $u \vee v = (u_i \vee v_i)_i$. Equivalently, $u \vee v$ is the supremum of the two vectors $u, v$ w.r.t. the entrywise ordering.} The condition $\lambda \vee \mu = 0$ represents the fact that the two weights $\lambda, \mu$ sums to $1$ in the tropical sense. These weights are implicitly nonnegative, since $\lambda, \mu \geq -\infty$. To summarize, the tropical segment consists of the tropical analogues of the convex combinations of the vectors $u$ and $v$. The set $\tsegm(u, v)$ is a polygonal curve, and the direction vector supporting every ordinary segment has its entries in $\{0,\pm 1\}$~\cite[proposition 3]{Develin2004}. We refer to~\Cref{fig:tropical_segments} for an illustration of tropical segments in dimension $2$. We say that a set $S \subset \T^n$ is \emph{tropically convex} if for all $u, v \in S$, the tropical segment $\tsegm(u, v)$ lies in $S$. The relevance
of tropical geometry to the study of log-limits can be readily seen by considering the following elementary properties, which hold for all $a,b\in\R$,
\begin{align}
  \lim_{t \to \infty}\log_t (t^a+t^b) = \max(a,b),\qquad
\log_t (t^a\times t^b)   \equiv a+b \enspace.\label{e:morphism}
\end{align}

\begin{figure}
	\begin{center}
	\begin{tikzpicture}
\tikzstyle{segm}=[red!50]	
	
\draw [step=1cm,gray!30,very thin] (0.5,0.5) grid (6.5,6.5) ;

\draw [segm, very thick] (4,6) -- (6,6) -- (6,4) ;
\draw [segm, very thick] (3,6) -- (1,4) -- (1,2) ;
\draw [segm, very thick] (2,1) -- (4,1) -- (6,3) ;

\fill [segm] (4,6) circle (2pt);
\fill [segm] (6,4) circle (2pt);
\fill [segm] (3,6) circle (2pt);
\fill [segm] (1,2) circle (2pt);
\fill [segm] (2,1) circle (2pt);
\fill [segm] (6,3) circle (2pt);

\end{tikzpicture}
	\end{center}
	\caption{The shapes of tropical segments in dimension 2.}\label{fig:tropical_segments}
\end{figure}

The following result is a straightforward consequence of~\eqref{e:morphism}.
\begin{proposition}
\label{prop:trop_conv}
The set $\Ktrop$ obtained in~\eqref{assmpt:assmpt_K_c}
as the limit of the family $\log_t \Kt$
is tropically convex. \hfill\qed
\end{proposition}
Albeit being seemingly technical, \Cref{assmpt:assmpt_K_c} is natural from the point of view of tropical geometry, and its statements are either automatically satisfied or easy to check for canonical classes of parametric convex programs,
including parametric linear programs, like the one we consider in our counter
example~\eqref{eq:cex}. 
Indeed, it is shown in \cite[Theorem~12]{ABGJ18} that if $\Kt$ is a parametric
family of linear programs,
then the log-limit $\Ktrop$ of the family $\Kt$
in the sense of Hausdorff metric always exists, with a $1/(\log t)$ rate of convergence,
\ie, $\distInf(\log_t \Kt, \Ktrop) = O(1/(\log t))$.
Moreover, general results, building on o-minimal geometry~\cite{alessandrini}
or on the model theory of real valued fields~\cite[Th.~3.1]{tropicalspectrahedra}
entail that $\Ktrop$ is always semilinear. 
The regularity condition is guaranteed when the leading exponents
of the coefficients of the parametric linear program are sufficiently generic.
Some of these results can be extended to parametric semialgebraic
sets~\cite{alessandrini}, \cite{tropicalspectrahedra}, so that
the scope of our approach is not limited to linear programming. 

We refer to \Cref{fig:log-limit} for an illustration of the convergence of $\log_t \Kt$ to the tropically convex set $\Ktrop$ in the case where the sets $\K_t$ are polyhedra.
The semilinear character of the log-limit is visible on this picture.
Note that, for parametric linear programs, the log-limit has
a well studied combinatorial structure, being a tropical polyhedron;
in particular, it is a finite union of alcoved polyhedra, i.e., ordinary
polyhedra defined by constraints of the form $x_i-x_j\leq a_{ij}$, $ x_i\leq b_i$,
$-x_i\leq b'_i$, with $a_{ij},b_i,b'_i\in \R\cup\{\infty\}$, see~\cite{Develin2004}.

 % dequantization
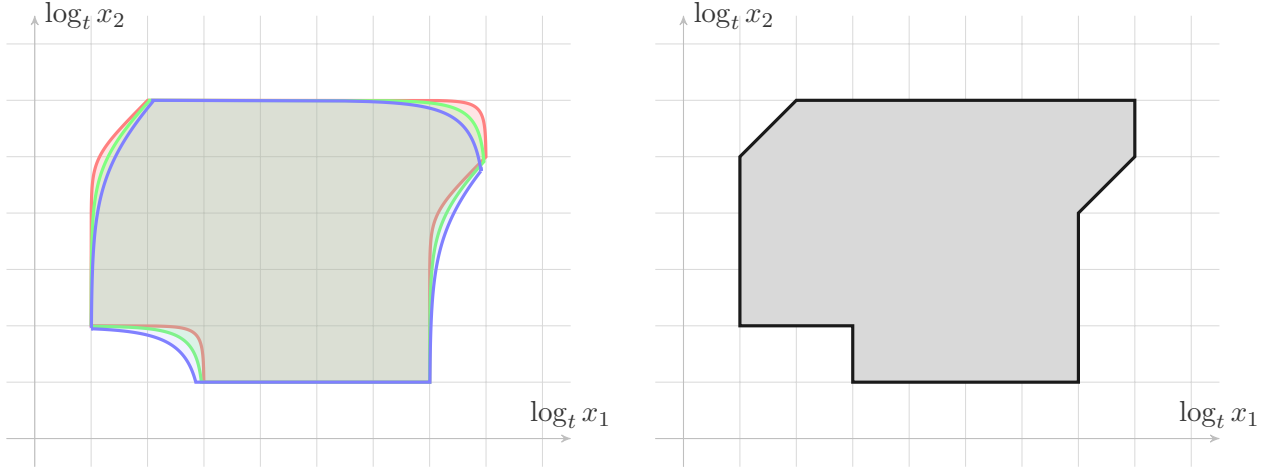
\begin{figure}[t!]

\begin{center}
\begin{tikzpicture}[scale=0.75]
\begin{scope}

	\draw[gray!30,very thin] (-0.5,-0.5) grid (9.5,7.5);
	\draw[gray!50,->,>=stealth'] (-0.5,0) -- (9.5,0) node[color=black!80,above] {$\log_t x_1$};
	\draw[gray!50,->,>=stealth'] (0,-0.5) -- (0,7.5) node[color=black!80,right] {$\log_t x_2$};
	
 % maslov-100.tikz
\filldraw[very thick, red!50, fill=red!50,fill opacity=0.2] (1,1.99998) -- (2.15377,1.99555) -- (2.30375,1.99102) -- (2.39162,1.9864) -- (2.454,1.98168) -- (2.5024,1.97685) -- (2.54195,1.97192) -- (2.5754,1.96687) -- (2.60438,1.9617) -- (2.62994,1.9564) -- (2.65281,1.95097) -- (2.67349,1.9454) -- (2.69238,1.93968) -- (2.70975,1.93381) -- (2.72584,1.92778) -- (2.74081,1.92158) -- (2.75482,1.91519) -- (2.76799,1.90861) -- (2.78039,1.90182) -- (2.79213,1.89481) -- (2.80327,1.88757) -- (2.81386,1.88008) -- (2.82396,1.87232) -- (2.83361,1.86428) -- (2.84285,1.85593) -- (2.85171,1.84724) -- (2.86023,1.83819) -- (2.86842,1.82874) -- (2.87632,1.81887) -- (2.88393,1.80852) -- (2.89129,1.79766) -- (2.89841,1.78622) -- (2.90531,1.77415) -- (2.91199,1.76137) -- (2.91847,1.74779) -- (2.92476,1.73331) -- (2.93088,1.71778) -- (2.93683,1.70107) -- (2.94262,1.68295) -- (2.94826,1.66319) -- (2.95375,1.64145) -- (2.95912,1.61728) -- (2.96435,1.59008) -- (2.96946,1.55899) -- (2.97445,1.52268) -- (2.97933,1.47906) -- (2.9841,1.42442) -- (2.98877,1.35123) -- (2.99334,1.24002) -- (2.99782,1) -- (2.99782,1) -- (6.1549,1) -- (6.30542,1) -- (6.39346,1) -- (6.45593,1) -- (6.50439,1) -- (6.54398,1) -- (6.57745,1) -- (6.60645,1) -- (6.63202,1) -- (6.6549,1) -- (6.6756,1) -- (6.69449,1) -- (6.71187,1) -- (6.72797,1) -- (6.74295,1) -- (6.75696,1) -- (6.77013,1) -- (6.78254,1) -- (6.79428,1) -- (6.80542,1) -- (6.81601,1) -- (6.82611,1) -- (6.83577,1) -- (6.84501,1) -- (6.85387,1) -- (6.86239,1) -- (6.87058,1) -- (6.87848,1) -- (6.8861,1) -- (6.89346,1) -- (6.90058,1) -- (6.90748,1) -- (6.91416,1) -- (6.92064,1) -- (6.92694,1) -- (6.93305,1) -- (6.939,1) -- (6.94479,1) -- (6.95043,1) -- (6.95593,1) -- (6.96129,1) -- (6.96653,1) -- (6.97164,1) -- (6.97663,1) -- (6.98151,1) -- (6.98628,1) -- (6.99095,1) -- (6.99552,1) -- (7,1) -- (7,1) -- (7.00258,3.03846) -- (7.00512,3.18896) -- (7.00764,3.27701) -- (7.01013,3.33947) -- (7.01259,3.38793) -- (7.01502,3.42752) -- (7.01743,3.46099) -- (7.0198,3.48999) -- (7.02216,3.51556) -- (7.02449,3.53844) -- (7.02679,3.55914) -- (7.02907,3.57803) -- (7.03132,3.59541) -- (7.03356,3.61151) -- (7.03577,3.62649) -- (7.03795,3.6405) -- (7.04012,3.65367) -- (7.04226,3.66608) -- (7.04439,3.67782) -- (7.04649,3.68896) -- (7.04857,3.69955) -- (7.05063,3.70965) -- (7.05268,3.7193) -- (7.0547,3.72855) -- (7.05671,3.73741) -- (7.05869,3.74593) -- (7.06066,3.75412) -- (7.06261,3.76202) -- (7.06455,3.76964) -- (7.06647,3.777) -- (7.06837,3.78412) -- (7.07025,3.79102) -- (7.07212,3.7977) -- (7.07397,3.80418) -- (7.0758,3.81047) -- (7.07763,3.81659) -- (7.07943,3.82254) -- (7.08122,3.82833) -- (7.083,3.83397) -- (7.08476,3.83947) -- (7.08651,3.84483) -- (7.08824,3.85006) -- (7.08996,3.85517) -- (7.09166,3.86017) -- (7.09336,3.86505) -- (7.09504,3.86982) -- (7.0967,3.87449) -- (7.09836,3.87906) -- (7.1,3.88354) -- (7.1,3.88354) -- (7.12234,3.93944) -- (7.14259,3.98386) -- (7.16112,4.02072) -- (7.17819,4.05222) -- (7.19401,4.07972) -- (7.20876,4.10413) -- (7.22257,4.12607) -- (7.23556,4.14599) -- (7.24781,4.16424) -- (7.25941,4.18108) -- (7.27041,4.1967) -- (7.28089,4.21127) -- (7.29089,4.22493) -- (7.30044,4.23778) -- (7.3096,4.24991) -- (7.31838,4.2614) -- (7.32682,4.27231) -- (7.33494,4.2827) -- (7.34278,4.29261) -- (7.35034,4.30209) -- (7.35764,4.31118) -- (7.36471,4.3199) -- (7.37155,4.32828) -- (7.37819,4.33635) -- (7.38463,4.34413) -- (7.39088,4.35165) -- (7.39696,4.35891) -- (7.40287,4.36594) -- (7.40863,4.37274) -- (7.41423,4.37934) -- (7.4197,4.38575) -- (7.42503,4.39197) -- (7.43024,4.39802) -- (7.43532,4.4039) -- (7.44028,4.40963) -- (7.44514,4.41521) -- (7.44989,4.42065) -- (7.45453,4.42596) -- (7.45908,4.43114) -- (7.46354,4.43621) -- (7.46791,4.44115) -- (7.47219,4.44599) -- (7.47639,4.45072) -- (7.4805,4.45535) -- (7.48455,4.45988) -- (7.48851,4.46432) -- (7.49241,4.46867) -- (7.49624,4.47294) -- (7.5,4.47712) -- (7.5,4.47712) -- (7.53625,4.51705) -- (7.5673,4.55076) -- (7.59447,4.57994) -- (7.61861,4.60565) -- (7.64033,4.62864) -- (7.66008,4.64943) -- (7.67818,4.6684) -- (7.69488,4.68585) -- (7.7104,4.70199) -- (7.72487,4.71702) -- (7.73845,4.73108) -- (7.75122,4.74428) -- (7.76328,4.75673) -- (7.77471,4.7685) -- (7.78557,4.77966) -- (7.79591,4.79028) -- (7.80578,4.8004) -- (7.81522,4.81007) -- (7.82427,4.81933) -- (7.83295,4.82821) -- (7.8413,4.83675) -- (7.84935,4.84496) -- (7.8571,4.85287) -- (7.86459,4.8605) -- (7.87183,4.86787) -- (7.87883,4.875) -- (7.88562,4.88191) -- (7.8922,4.8886) -- (7.89858,4.89509) -- (7.90479,4.90139) -- (7.91082,4.90752) -- (7.91669,4.91348) -- (7.9224,4.91927) -- (7.92797,4.92492) -- (7.9334,4.93043) -- (7.93869,4.93579) -- (7.94386,4.94103) -- (7.94891,4.94615) -- (7.95385,4.95114) -- (7.95867,4.95603) -- (7.96339,4.96081) -- (7.96801,4.96548) -- (7.97254,4.97006) -- (7.97697,4.97454) -- (7.98131,4.97893) -- (7.98557,4.98323) -- (7.98974,4.98745) -- (7.99384,4.99159) -- (7.99786,4.99566) -- (7.99786,4.99566) -- (7.99338,5.23863) -- (7.98881,5.35043) -- (7.98414,5.42387) -- (7.97937,5.47865) -- (7.97449,5.52235) -- (7.9695,5.55872) -- (7.96439,5.58986) -- (7.95915,5.61709) -- (7.95379,5.64128) -- (7.9483,5.66305) -- (7.94265,5.68283) -- (7.93686,5.70096) -- (7.93091,5.71769) -- (7.9248,5.73322) -- (7.9185,5.74772) -- (7.91202,5.76131) -- (7.90534,5.7741) -- (7.89844,5.78617) -- (7.89132,5.79761) -- (7.88396,5.80848) -- (7.87634,5.81883) -- (7.86844,5.82871) -- (7.86025,5.83816) -- (7.85173,5.84721) -- (7.84287,5.85591) -- (7.83363,5.86426) -- (7.82397,5.87231) -- (7.81387,5.88007) -- (7.80328,5.88756) -- (7.79214,5.89481) -- (7.7804,5.90182) -- (7.76799,5.9086) -- (7.75482,5.91519) -- (7.74081,5.92158) -- (7.72583,5.92779) -- (7.70973,5.93382) -- (7.69235,5.93969) -- (7.67346,5.94541) -- (7.65276,5.95098) -- (7.62988,5.95641) -- (7.60431,5.96171) -- (7.57531,5.96688) -- (7.54184,5.97193) -- (7.50225,5.97687) -- (7.45379,5.9817) -- (7.39132,5.98642) -- (7.30328,5.99104) -- (7.15276,5.99557) -- (2.00218,6) -- (2.00216,6) -- (1.99813,5.99593) -- (1.99402,5.99178) -- (1.98984,5.98755) -- (1.98557,5.98324) -- (1.98122,5.97884) -- (1.97678,5.97435) -- (1.97224,5.96976) -- (1.96761,5.96507) -- (1.96288,5.96029) -- (1.95804,5.95539) -- (1.95309,5.95038) -- (1.94803,5.94525) -- (1.94284,5.94) -- (1.93753,5.93462) -- (1.93209,5.9291) -- (1.9265,5.92343) -- (1.92077,5.91762) -- (1.91488,5.91164) -- (1.90883,5.9055) -- (1.9026,5.89918) -- (1.8962,5.89266) -- (1.88959,5.88595) -- (1.88278,5.87902) -- (1.87575,5.87187) -- (1.86848,5.86447) -- (1.86096,5.85681) -- (1.85318,5.84886) -- (1.8451,5.84062) -- (1.83671,5.83205) -- (1.82798,5.82313) -- (1.81889,5.81383) -- (1.8094,5.80411) -- (1.79947,5.79394) -- (1.78907,5.78326) -- (1.77815,5.77203) -- (1.76665,5.76019) -- (1.75451,5.74767) -- (1.74164,5.73438) -- (1.72797,5.72023) -- (1.71337,5.70509) -- (1.69773,5.68881) -- (1.68086,5.67121) -- (1.66258,5.65206) -- (1.64262,5.63106) -- (1.62063,5.6078) -- (1.59616,5.58175) -- (1.56859,5.55215) -- (1.53699,5.51786) -- (1.5,5.47712) -- (1.5,5.47712) -- (1.49624,5.47294) -- (1.49241,5.46867) -- (1.48851,5.46432) -- (1.48455,5.45988) -- (1.4805,5.45535) -- (1.47639,5.45072) -- (1.47219,5.44599) -- (1.46791,5.44115) -- (1.46354,5.43621) -- (1.45908,5.43114) -- (1.45453,5.42596) -- (1.44989,5.42065) -- (1.44514,5.41521) -- (1.44028,5.40963) -- (1.43532,5.4039) -- (1.43024,5.39802) -- (1.42503,5.39197) -- (1.4197,5.38575) -- (1.41423,5.37934) -- (1.40863,5.37274) -- (1.40287,5.36594) -- (1.39696,5.35891) -- (1.39088,5.35165) -- (1.38463,5.34413) -- (1.37819,5.33635) -- (1.37155,5.32828) -- (1.36471,5.3199) -- (1.35764,5.31118) -- (1.35034,5.30209) -- (1.34278,5.29261) -- (1.33494,5.2827) -- (1.32682,5.27231) -- (1.31838,5.2614) -- (1.3096,5.24991) -- (1.30044,5.23778) -- (1.29089,5.22493) -- (1.28089,5.21127) -- (1.27041,5.1967) -- (1.25941,5.18108) -- (1.24781,5.16424) -- (1.23556,5.14599) -- (1.22257,5.12607) -- (1.20876,5.10413) -- (1.19401,5.07972) -- (1.17819,5.05222) -- (1.16112,5.02072) -- (1.14259,4.98386) -- (1.12234,4.93944) -- (1.1,4.88354) -- (1.1,4.88354) -- (1.09836,4.87906) -- (1.0967,4.87449) -- (1.09504,4.86982) -- (1.09336,4.86505) -- (1.09166,4.86017) -- (1.08996,4.85517) -- (1.08824,4.85006) -- (1.08651,4.84483) -- (1.08476,4.83947) -- (1.083,4.83397) -- (1.08122,4.82833) -- (1.07943,4.82254) -- (1.07763,4.81659) -- (1.0758,4.81047) -- (1.07397,4.80418) -- (1.07212,4.7977) -- (1.07025,4.79102) -- (1.06837,4.78412) -- (1.06647,4.777) -- (1.06455,4.76964) -- (1.06261,4.76202) -- (1.06066,4.75412) -- (1.05869,4.74593) -- (1.05671,4.73741) -- (1.0547,4.72855) -- (1.05268,4.7193) -- (1.05063,4.70965) -- (1.04857,4.69955) -- (1.04649,4.68896) -- (1.04439,4.67782) -- (1.04226,4.66608) -- (1.04012,4.65367) -- (1.03795,4.6405) -- (1.03577,4.62649) -- (1.03356,4.61151) -- (1.03132,4.59541) -- (1.02907,4.57803) -- (1.02679,4.55914) -- (1.02449,4.53844) -- (1.02216,4.51556) -- (1.0198,4.48999) -- (1.01743,4.46099) -- (1.01502,4.42752) -- (1.01259,4.38793) -- (1.01013,4.33947) -- (1.00764,4.27701) -- (1.00512,4.18896) -- (1.00258,4.03846) -- (1,1.99998) -- cycle;

 % maslov-10.tikz
\filldraw[very thick, green!50, fill=green!50,fill opacity=0.2] (1,1.99123) -- (1.44968,1.98319) -- (1.66583,1.97501) -- (1.80949,1.96667) -- (1.91726,1.95816) -- (2.00353,1.94949) -- (2.07547,1.94064) -- (2.13717,1.9316) -- (2.19119,1.92238) -- (2.23922,1.91295) -- (2.28247,1.90331) -- (2.3218,1.89346) -- (2.35786,1.88337) -- (2.39115,1.87305) -- (2.42207,1.86247) -- (2.45094,1.85163) -- (2.47801,1.84052) -- (2.50348,1.82911) -- (2.52755,1.81739) -- (2.55035,1.80535) -- (2.57202,1.79296) -- (2.59265,1.78021) -- (2.61235,1.76708) -- (2.6312,1.75354) -- (2.64926,1.73956) -- (2.66659,1.72511) -- (2.68327,1.71017) -- (2.69932,1.69469) -- (2.71481,1.67865) -- (2.72976,1.66199) -- (2.74421,1.64466) -- (2.7582,1.62661) -- (2.77175,1.60778) -- (2.78489,1.5881) -- (2.79765,1.56748) -- (2.81004,1.54584) -- (2.82209,1.52306) -- (2.83381,1.49901) -- (2.84522,1.47356) -- (2.85635,1.44652) -- (2.86719,1.41769) -- (2.87777,1.38681) -- (2.8881,1.35356) -- (2.89819,1.31755) -- (2.90805,1.27829) -- (2.91769,1.23512) -- (2.92712,1.18718) -- (2.93635,1.13328) -- (2.94539,1.07173) -- (2.95424,1) -- (2.95424,1) -- (5.31211,1) -- (5.61218,1) -- (5.78796,1) -- (5.91274,1) -- (6.00955,1) -- (6.08867,1) -- (6.15557,1) -- (6.21353,1) -- (6.26465,1) -- (6.31039,1) -- (6.35177,1) -- (6.38954,1) -- (6.42429,1) -- (6.45646,1) -- (6.48642,1) -- (6.51444,1) -- (6.54076,1) -- (6.56558,1) -- (6.58905,1) -- (6.61132,1) -- (6.63251,1) -- (6.65271,1) -- (6.67201,1) -- (6.69049,1) -- (6.70822,1) -- (6.72525,1) -- (6.74163,1) -- (6.75743,1) -- (6.77266,1) -- (6.78738,1) -- (6.80162,1) -- (6.81541,1) -- (6.82877,1) -- (6.84173,1) -- (6.85432,1) -- (6.86655,1) -- (6.87845,1) -- (6.89003,1) -- (6.90131,1) -- (6.91231,1) -- (6.92303,1) -- (6.93349,1) -- (6.94371,1) -- (6.9537,1) -- (6.96345,1) -- (6.973,1) -- (6.98234,1) -- (6.99148,1) -- (7.00043,1) -- (7.00043,1) -- (7.00271,1.79684) -- (7.00498,2.06174) -- (7.00723,2.22509) -- (7.00947,2.34351) -- (7.01171,2.43646) -- (7.01392,2.51299) -- (7.01613,2.57803) -- (7.01833,2.63458) -- (7.02052,2.68461) -- (7.02269,2.72947) -- (7.02486,2.77013) -- (7.02701,2.8073) -- (7.02915,2.84154) -- (7.03128,2.87328) -- (7.03341,2.90286) -- (7.03552,2.93054) -- (7.03762,2.95657) -- (7.03971,2.98113) -- (7.04179,3.00437) -- (7.04386,3.02643) -- (7.04592,3.04743) -- (7.04798,3.06745) -- (7.05002,3.0866) -- (7.05205,3.10493) -- (7.05407,3.12253) -- (7.05609,3.13943) -- (7.05809,3.15571) -- (7.06009,3.17139) -- (7.06207,3.18653) -- (7.06405,3.20116) -- (7.06602,3.21531) -- (7.06798,3.22902) -- (7.06993,3.24231) -- (7.07187,3.2552) -- (7.0738,3.26772) -- (7.07573,3.27989) -- (7.07764,3.29172) -- (7.07955,3.30325) -- (7.08145,3.31447) -- (7.08334,3.32542) -- (7.08522,3.33609) -- (7.0871,3.34651) -- (7.08897,3.35668) -- (7.09082,3.36662) -- (7.09268,3.37634) -- (7.09452,3.38585) -- (7.09635,3.39515) -- (7.09818,3.40425) -- (7.1,3.41317) -- (7.1,3.41317) -- (7.1132,3.47388) -- (7.12601,3.52713) -- (7.13845,3.57456) -- (7.15054,3.61732) -- (7.16231,3.65624) -- (7.17377,3.69196) -- (7.18493,3.72496) -- (7.19581,3.75563) -- (7.20643,3.78428) -- (7.21679,3.81115) -- (7.22691,3.83646) -- (7.2368,3.86037) -- (7.24647,3.88303) -- (7.25593,3.90458) -- (7.26519,3.9251) -- (7.27426,3.94469) -- (7.28314,3.96344) -- (7.29184,3.98142) -- (7.30037,3.99868) -- (7.30874,4.01528) -- (7.31694,4.03127) -- (7.325,4.04669) -- (7.33291,4.06158) -- (7.34068,4.07598) -- (7.34831,4.08991) -- (7.35581,4.10342) -- (7.36318,4.11651) -- (7.37043,4.12922) -- (7.37756,4.14157) -- (7.38458,4.15358) -- (7.39148,4.16527) -- (7.39828,4.17665) -- (7.40497,4.18774) -- (7.41156,4.19855) -- (7.41805,4.2091) -- (7.42444,4.2194) -- (7.43074,4.22946) -- (7.43696,4.2393) -- (7.44308,4.24891) -- (7.44912,4.25832) -- (7.45508,4.26753) -- (7.46095,4.27655) -- (7.46675,4.28538) -- (7.47247,4.29404) -- (7.47812,4.30253) -- (7.48369,4.31085) -- (7.4892,4.31902) -- (7.49463,4.32704) -- (7.5,4.33491) -- (7.5,4.33491) -- (7.51655,4.35891) -- (7.5325,4.38165) -- (7.54788,4.40326) -- (7.56273,4.42385) -- (7.57709,4.4435) -- (7.59099,4.4623) -- (7.60446,4.48032) -- (7.61753,4.49763) -- (7.63022,4.51427) -- (7.64254,4.5303) -- (7.65452,4.54575) -- (7.66619,4.56068) -- (7.67754,4.57511) -- (7.68861,4.58907) -- (7.6994,4.6026) -- (7.70993,4.61572) -- (7.72022,4.62846) -- (7.73026,4.64083) -- (7.74008,4.65286) -- (7.74968,4.66457) -- (7.75907,4.67597) -- (7.76826,4.68708) -- (7.77727,4.69791) -- (7.78609,4.70847) -- (7.79473,4.71879) -- (7.80321,4.72887) -- (7.81152,4.73871) -- (7.81968,4.74834) -- (7.82768,4.75777) -- (7.83555,4.76699) -- (7.84327,4.77601) -- (7.85085,4.78486) -- (7.85831,4.79353) -- (7.86564,4.80203) -- (7.87285,4.81036) -- (7.87994,4.81854) -- (7.88692,4.82657) -- (7.89379,4.83445) -- (7.90055,4.84219) -- (7.90721,4.8498) -- (7.91376,4.85727) -- (7.92022,4.86462) -- (7.92659,4.87185) -- (7.93286,4.87896) -- (7.93904,4.88595) -- (7.94514,4.89283) -- (7.95115,4.89961) -- (7.95708,4.90628) -- (7.96293,4.91285) -- (7.96293,4.91285) -- (7.95397,5.00241) -- (7.94483,5.07662) -- (7.93549,5.13998) -- (7.92595,5.19527) -- (7.91619,5.2443) -- (7.9062,5.28836) -- (7.89598,5.32835) -- (7.88552,5.36497) -- (7.87479,5.39874) -- (7.8638,5.43008) -- (7.85252,5.4593) -- (7.84093,5.48668) -- (7.82904,5.51244) -- (7.8168,5.53675) -- (7.80421,5.55978) -- (7.79125,5.58164) -- (7.77788,5.60246) -- (7.76409,5.62233) -- (7.74985,5.64132) -- (7.73513,5.65952) -- (7.71989,5.67699) -- (7.7041,5.69378) -- (7.68771,5.70995) -- (7.67067,5.72554) -- (7.65294,5.74058) -- (7.63446,5.75512) -- (7.61516,5.7692) -- (7.59495,5.78283) -- (7.57376,5.79604) -- (7.55149,5.80887) -- (7.52801,5.82132) -- (7.50318,5.83343) -- (7.47685,5.84521) -- (7.44883,5.85668) -- (7.41886,5.86786) -- (7.38668,5.87875) -- (7.35192,5.88938) -- (7.31413,5.89975) -- (7.27273,5.90988) -- (7.22698,5.91979) -- (7.17583,5.92947) -- (7.11783,5.93893) -- (7.05089,5.9482) -- (6.97171,5.95728) -- (6.8748,5.96616) -- (6.74986,5.97487) -- (6.57377,5.98341) -- (6.27276,5.99178) -- (2.03743,6) -- (2.03743,5.99564) -- (2.03109,5.98865) -- (2.02466,5.98155) -- (2.01813,5.97434) -- (2.0115,5.967) -- (2.00477,5.95954) -- (1.99793,5.95195) -- (1.99099,5.94422) -- (1.98393,5.93635) -- (1.97675,5.92833) -- (1.96946,5.92017) -- (1.96204,5.91185) -- (1.95449,5.90336) -- (1.9468,5.89471) -- (1.93898,5.88588) -- (1.93102,5.87687) -- (1.9229,5.86767) -- (1.91463,5.85827) -- (1.9062,5.84866) -- (1.89761,5.83883) -- (1.88884,5.82877) -- (1.87989,5.81848) -- (1.87075,5.80794) -- (1.86142,5.79713) -- (1.85188,5.78605) -- (1.84212,5.77468) -- (1.83214,5.763) -- (1.82193,5.751) -- (1.81147,5.73866) -- (1.80075,5.72596) -- (1.78977,5.71287) -- (1.77849,5.69938) -- (1.76692,5.68545) -- (1.75503,5.67107) -- (1.7428,5.65619) -- (1.73022,5.64078) -- (1.71727,5.62481) -- (1.70391,5.60823) -- (1.69014,5.59099) -- (1.67591,5.57303) -- (1.6612,5.55431) -- (1.64597,5.53473) -- (1.63019,5.51424) -- (1.61382,5.49273) -- (1.5968,5.47009) -- (1.57909,5.44622) -- (1.56063,5.42095) -- (1.54134,5.39412) -- (1.52116,5.36552) -- (1.5,5.33491) -- (1.5,5.33491) -- (1.49463,5.32704) -- (1.4892,5.31902) -- (1.48369,5.31085) -- (1.47812,5.30253) -- (1.47247,5.29404) -- (1.46675,5.28538) -- (1.46095,5.27655) -- (1.45508,5.26753) -- (1.44912,5.25832) -- (1.44308,5.24891) -- (1.43696,5.2393) -- (1.43074,5.22946) -- (1.42444,5.2194) -- (1.41805,5.2091) -- (1.41156,5.19855) -- (1.40497,5.18774) -- (1.39828,5.17665) -- (1.39148,5.16527) -- (1.38458,5.15358) -- (1.37756,5.14157) -- (1.37043,5.12922) -- (1.36318,5.11651) -- (1.35581,5.10342) -- (1.34831,5.08991) -- (1.34068,5.07598) -- (1.33291,5.06158) -- (1.325,5.04669) -- (1.31694,5.03127) -- (1.30874,5.01528) -- (1.30037,4.99868) -- (1.29184,4.98142) -- (1.28314,4.96344) -- (1.27426,4.94469) -- (1.26519,4.9251) -- (1.25593,4.90458) -- (1.24647,4.88303) -- (1.2368,4.86037) -- (1.22691,4.83646) -- (1.21679,4.81115) -- (1.20643,4.78428) -- (1.19581,4.75563) -- (1.18493,4.72496) -- (1.17377,4.69196) -- (1.16231,4.65624) -- (1.15054,4.61732) -- (1.13845,4.57456) -- (1.12601,4.52713) -- (1.1132,4.47388) -- (1.1,4.41317) -- (1.1,4.41317) -- (1.09818,4.40425) -- (1.09635,4.39515) -- (1.09452,4.38584) -- (1.09267,4.37634) -- (1.09082,4.36662) -- (1.08896,4.35668) -- (1.0871,4.3465) -- (1.08522,4.33608) -- (1.08334,4.32541) -- (1.08145,4.31446) -- (1.07955,4.30324) -- (1.07764,4.29171) -- (1.07573,4.27987) -- (1.0738,4.2677) -- (1.07187,4.25518) -- (1.06993,4.24229) -- (1.06798,4.229) -- (1.06602,4.21529) -- (1.06405,4.20114) -- (1.06207,4.18651) -- (1.06008,4.17137) -- (1.05809,4.15568) -- (1.05608,4.1394) -- (1.05407,4.12249) -- (1.05205,4.1049) -- (1.05001,4.08656) -- (1.04797,4.06741) -- (1.04592,4.04738) -- (1.04386,4.02639) -- (1.04179,4.00432) -- (1.03971,3.98107) -- (1.03761,3.95651) -- (1.03551,3.93048) -- (1.0334,3.90278) -- (1.03128,3.8732) -- (1.02915,3.84145) -- (1.027,3.8072) -- (1.02485,3.77001) -- (1.02268,3.72934) -- (1.02051,3.68447) -- (1.01832,3.63441) -- (1.01613,3.57783) -- (1.01392,3.51275) -- (1.0117,3.43618) -- (1.00947,3.34315) -- (1.00722,3.2246) -- (1.00497,3.06101) -- (1.0027,2.79548) -- (1.00043,1.99123) -- cycle;

 % maslov-5.tikz
\filldraw[very thick, blue!50, fill=blue!50,fill opacity=0.1] (1,1.94819) -- (1.20177,1.93819) -- (1.35384,1.92802) -- (1.47591,1.91768) -- (1.57789,1.90717) -- (1.66546,1.89648) -- (1.74221,1.8856) -- (1.81051,1.87453) -- (1.87204,1.86325) -- (1.92802,1.85177) -- (1.97937,1.84007) -- (2.0268,1.82815) -- (2.07087,1.81599) -- (2.11201,1.80359) -- (2.1506,1.79094) -- (2.18693,1.77802) -- (2.22126,1.76483) -- (2.25379,1.75136) -- (2.2847,1.73758) -- (2.31414,1.7235) -- (2.34225,1.70909) -- (2.36914,1.69433) -- (2.39492,1.67922) -- (2.41967,1.66372) -- (2.44348,1.64784) -- (2.4664,1.63153) -- (2.48851,1.61479) -- (2.50986,1.59758) -- (2.5305,1.57988) -- (2.55048,1.56167) -- (2.56983,1.5429) -- (2.5886,1.52355) -- (2.60682,1.50358) -- (2.62452,1.48294) -- (2.64173,1.4616) -- (2.65847,1.43949) -- (2.67478,1.41657) -- (2.69067,1.39277) -- (2.70616,1.36802) -- (2.72128,1.34225) -- (2.73604,1.31536) -- (2.75045,1.28726) -- (2.76454,1.25782) -- (2.77831,1.22692) -- (2.79179,1.1944) -- (2.80498,1.16008) -- (2.8179,1.12376) -- (2.83055,1.08518) -- (2.84295,1.04405) -- (2.85511,1) -- (2.85511,1) -- (4.62324,1) -- (5.0356,1) -- (5.2813,1) -- (5.45691,1) -- (5.59367,1) -- (5.70569,1) -- (5.80057,1) -- (5.88285,1) -- (5.95551,1) -- (6.02055,1) -- (6.07942,1) -- (6.1332,1) -- (6.18269,1) -- (6.22852,1) -- (6.27121,1) -- (6.31115,1) -- (6.34868,1) -- (6.38407,1) -- (6.41755,1) -- (6.44932,1) -- (6.47954,1) -- (6.50837,1) -- (6.53591,1) -- (6.56228,1) -- (6.58758,1) -- (6.61189,1) -- (6.63529,1) -- (6.65784,1) -- (6.67959,1) -- (6.70061,1) -- (6.72094,1) -- (6.74063,1) -- (6.75972,1) -- (6.77823,1) -- (6.79621,1) -- (6.81368,1) -- (6.83068,1) -- (6.84722,1) -- (6.86333,1) -- (6.87904,1) -- (6.89436,1) -- (6.90931,1) -- (6.92391,1) -- (6.93817,1) -- (6.95212,1) -- (6.96575,1) -- (6.9791,1) -- (6.99216,1) -- (7.00496,1) -- (7.00495,1) -- (7.00704,1.22008) -- (7.00913,1.38227) -- (7.01121,1.51076) -- (7.01328,1.61719) -- (7.01534,1.70802) -- (7.0174,1.78725) -- (7.01945,1.85751) -- (7.0215,1.92063) -- (7.02353,1.97792) -- (7.02557,2.03037) -- (7.02759,2.07874) -- (7.02961,2.12361) -- (7.03162,2.16546) -- (7.03362,2.20467) -- (7.03562,2.24155) -- (7.03761,2.27636) -- (7.0396,2.30933) -- (7.04158,2.34063) -- (7.04355,2.37043) -- (7.04552,2.39887) -- (7.04748,2.42606) -- (7.04943,2.45212) -- (7.05138,2.47712) -- (7.05332,2.50116) -- (7.05526,2.5243) -- (7.05719,2.54661) -- (7.05911,2.56814) -- (7.06103,2.58896) -- (7.06294,2.6091) -- (7.06485,2.62861) -- (7.06675,2.64752) -- (7.06864,2.66588) -- (7.07053,2.68371) -- (7.07242,2.70104) -- (7.07429,2.7179) -- (7.07617,2.73431) -- (7.07803,2.75031) -- (7.07989,2.7659) -- (7.08175,2.78111) -- (7.0836,2.79596) -- (7.08544,2.81046) -- (7.08728,2.82462) -- (7.08911,2.83848) -- (7.09094,2.85203) -- (7.09276,2.86529) -- (7.09458,2.87828) -- (7.09639,2.891) -- (7.0982,2.90346) -- (7.1,2.91568) -- (7.1,2.91568) -- (7.11135,2.98834) -- (7.1225,3.05338) -- (7.13346,3.11226) -- (7.14422,3.16604) -- (7.1548,3.21553) -- (7.16521,3.26137) -- (7.17544,3.30406) -- (7.18551,3.344) -- (7.19541,3.38153) -- (7.20516,3.41692) -- (7.21476,3.4504) -- (7.22422,3.48217) -- (7.23353,3.5124) -- (7.2427,3.54122) -- (7.25174,3.56877) -- (7.26066,3.59514) -- (7.26944,3.62044) -- (7.2781,3.64475) -- (7.28665,3.66815) -- (7.29508,3.6907) -- (7.30339,3.71245) -- (7.3116,3.73347) -- (7.31969,3.75381) -- (7.32769,3.7735) -- (7.33558,3.79258) -- (7.34338,3.81109) -- (7.35107,3.82907) -- (7.35868,3.84655) -- (7.36619,3.86354) -- (7.37361,3.88009) -- (7.38094,3.8962) -- (7.38819,3.91191) -- (7.39535,3.92723) -- (7.40244,3.94218) -- (7.40944,3.95678) -- (7.41636,3.97104) -- (7.42321,3.98498) -- (7.42999,3.99862) -- (7.43669,4.01197) -- (7.44332,4.02503) -- (7.44988,4.03783) -- (7.45637,4.05036) -- (7.46279,4.06265) -- (7.46915,4.0747) -- (7.47544,4.08653) -- (7.48167,4.09813) -- (7.48784,4.10951) -- (7.49395,4.1207) -- (7.5,4.13168) -- (7.5,4.13168) -- (7.51178,4.15283) -- (7.52334,4.17329) -- (7.53469,4.19309) -- (7.54584,4.21228) -- (7.55679,4.2309) -- (7.56755,4.24897) -- (7.57813,4.26653) -- (7.58853,4.28361) -- (7.59876,4.30024) -- (7.60882,4.31643) -- (7.61872,4.33221) -- (7.62847,4.3476) -- (7.63807,4.36261) -- (7.64752,4.37727) -- (7.65683,4.3916) -- (7.666,4.4056) -- (7.67504,4.41929) -- (7.68395,4.43269) -- (7.69273,4.4458) -- (7.70139,4.45865) -- (7.70993,4.47123) -- (7.71836,4.48356) -- (7.72667,4.49566) -- (7.73488,4.50752) -- (7.74297,4.51916) -- (7.75096,4.53059) -- (7.75886,4.54181) -- (7.76665,4.55283) -- (7.77434,4.56366) -- (7.78194,4.5743) -- (7.78945,4.58476) -- (7.79687,4.59505) -- (7.8042,4.60518) -- (7.81145,4.61514) -- (7.81861,4.62494) -- (7.8257,4.63459) -- (7.8327,4.6441) -- (7.83962,4.65346) -- (7.84647,4.66268) -- (7.85324,4.67177) -- (7.85994,4.68072) -- (7.86657,4.68955) -- (7.87313,4.69826) -- (7.87962,4.70684) -- (7.88604,4.71531) -- (7.8924,4.72366) -- (7.89869,4.7319) -- (7.90492,4.74004) -- (7.91109,4.74807) -- (7.91109,4.74797) -- (7.89827,4.82537) -- (7.88519,4.89418) -- (7.87183,4.95613) -- (7.85818,5.01246) -- (7.84422,5.0641) -- (7.82993,5.11178) -- (7.81531,5.15606) -- (7.80034,5.19739) -- (7.785,5.23615) -- (7.76927,5.27263) -- (7.75314,5.30708) -- (7.73657,5.33972) -- (7.71955,5.37074) -- (7.70204,5.40028) -- (7.68404,5.42848) -- (7.66549,5.45545) -- (7.64637,5.4813) -- (7.62665,5.50612) -- (7.60628,5.52999) -- (7.58522,5.55297) -- (7.56342,5.57513) -- (7.54082,5.59653) -- (7.51738,5.61722) -- (7.49301,5.63724) -- (7.46765,5.65663) -- (7.44121,5.67544) -- (7.4136,5.6937) -- (7.3847,5.71143) -- (7.35439,5.72867) -- (7.32253,5.74545) -- (7.28894,5.76178) -- (7.25344,5.7777) -- (7.21578,5.79322) -- (7.17569,5.80836) -- (7.13283,5.82314) -- (7.0868,5.83758) -- (7.03708,5.85169) -- (6.98304,5.86549) -- (6.92385,5.87898) -- (6.85841,5.89219) -- (6.78526,5.90513) -- (6.70234,5.9178) -- (6.60663,5.93021) -- (6.49343,5.94239) -- (6.35491,5.95433) -- (6.17638,5.96604) -- (5.92489,5.97754) -- (5.49551,5.98883) -- (2.09222,5.99992) -- (2.09222,5.97464) -- (2.08438,5.96515) -- (2.07644,5.95551) -- (2.06839,5.94573) -- (2.06024,5.93579) -- (2.05199,5.92568) -- (2.04362,5.91541) -- (2.03513,5.90497) -- (2.02653,5.89434) -- (2.01781,5.88354) -- (2.00897,5.87254) -- (1.99999,5.86134) -- (1.99089,5.84994) -- (1.98165,5.83833) -- (1.97227,5.82649) -- (1.96274,5.81442) -- (1.95307,5.80212) -- (1.94325,5.78956) -- (1.93326,5.77675) -- (1.92312,5.76367) -- (1.9128,5.7503) -- (1.90232,5.73665) -- (1.89165,5.72268) -- (1.88079,5.70839) -- (1.86975,5.69377) -- (1.8585,5.6788) -- (1.84704,5.66345) -- (1.83537,5.64772) -- (1.82348,5.63158) -- (1.81135,5.61501) -- (1.79899,5.59798) -- (1.78637,5.58047) -- (1.77349,5.56246) -- (1.76034,5.54391) -- (1.7469,5.52479) -- (1.73317,5.50506) -- (1.71912,5.48468) -- (1.70475,5.46361) -- (1.69005,5.4418) -- (1.67498,5.4192) -- (1.65954,5.39575) -- (1.64371,5.37137) -- (1.62746,5.346) -- (1.61077,5.31955) -- (1.59363,5.29193) -- (1.576,5.26301) -- (1.55785,5.23269) -- (1.53916,5.20081) -- (1.51989,5.16721) -- (1.5,5.13168) -- (1.5,5.13168) -- (1.49395,5.1207) -- (1.48784,5.10951) -- (1.48167,5.09813) -- (1.47544,5.08653) -- (1.46915,5.0747) -- (1.46279,5.06265) -- (1.45637,5.05036) -- (1.44988,5.03783) -- (1.44332,5.02503) -- (1.43669,5.01197) -- (1.42999,4.99862) -- (1.42321,4.98498) -- (1.41636,4.97104) -- (1.40944,4.95678) -- (1.40244,4.94218) -- (1.39535,4.92723) -- (1.38819,4.91191) -- (1.38094,4.8962) -- (1.37361,4.88009) -- (1.36619,4.86354) -- (1.35868,4.84655) -- (1.35107,4.82907) -- (1.34338,4.81109) -- (1.33558,4.79258) -- (1.32769,4.7735) -- (1.31969,4.75381) -- (1.3116,4.73347) -- (1.30339,4.71245) -- (1.29508,4.6907) -- (1.28665,4.66815) -- (1.2781,4.64475) -- (1.26944,4.62044) -- (1.26066,4.59514) -- (1.25174,4.56877) -- (1.2427,4.54122) -- (1.23353,4.5124) -- (1.22422,4.48217) -- (1.21476,4.4504) -- (1.20516,4.41692) -- (1.19541,4.38153) -- (1.18551,4.344) -- (1.17544,4.30406) -- (1.16521,4.26137) -- (1.1548,4.21553) -- (1.14422,4.16604) -- (1.13346,4.11226) -- (1.1225,4.05338) -- (1.11135,3.98834) -- (1.1,3.91568) -- (1.1,3.91568) -- (1.09819,3.90341) -- (1.09638,3.8909) -- (1.09456,3.87813) -- (1.09274,3.86509) -- (1.09091,3.85177) -- (1.08907,3.83816) -- (1.08723,3.82425) -- (1.08538,3.81002) -- (1.08353,3.79545) -- (1.08168,3.78053) -- (1.07981,3.76525) -- (1.07795,3.74958) -- (1.07607,3.73351) -- (1.07419,3.717) -- (1.07231,3.70005) -- (1.07042,3.68263) -- (1.06852,3.6647) -- (1.06662,3.64623) -- (1.06471,3.62721) -- (1.0628,3.60758) -- (1.06088,3.58731) -- (1.05895,3.56635) -- (1.05702,3.54467) -- (1.05508,3.5222) -- (1.05314,3.49889) -- (1.05119,3.47467) -- (1.04923,3.44946) -- (1.04727,3.42319) -- (1.0453,3.39577) -- (1.04332,3.36707) -- (1.04134,3.33698) -- (1.03936,3.30537) -- (1.03736,3.27205) -- (1.03536,3.23685) -- (1.03335,3.19954) -- (1.03134,3.15984) -- (1.02932,3.11743) -- (1.02729,3.07191) -- (1.02526,3.02279) -- (1.02322,2.96945) -- (1.02118,2.9111) -- (1.01912,2.8467) -- (1.01706,2.77484) -- (1.015,2.69357) -- (1.01292,2.60005) -- (1.01084,2.48991) -- (1.00875,2.35595) -- (1.00666,2.18494) -- (1.00456,1.94819) -- cycle;

\end{scope}

\begin{scope}[shift={(11.5,0)}]
	\draw[gray!30,very thin] (-0.5,-0.5) grid (9.5,7.5);
	\draw[gray!50,->,>=stealth'] (-0.5,0) -- (9.5,0) node[color=black!80,above] {$\log_t x_1$};
	\draw[gray!50,->,>=stealth'] (0,-0.5) -- (0,7.5) node[color=black!80,right] {$\log_t x_2$};	

\filldraw[very thick, black!90, fill=gray!30] (1,2) -- (3,2) -- (3,1) -- (7,1) -- (7,4) -- (8,5) -- (8,6) -- (2,6) -- (1,5) -- cycle;
\end{scope}
\end{tikzpicture}
\end{center}	

\caption{Left: the image under the map $\log_t$ of the polyhedron defined by the constraints $x_1 + t x_2 \geq t^3$, $x_2 \geq t^{-10} x_1 + t$, $x_2 + t^4 \geq t^{-3} x_1$, $t^8 \geq x_1 + t^2 x_2$, $t^4 x_1 \geq x_2 + t^5$, for $t = 5$ (blue), $10$ (green) and $100$ (red). Right: the log-limit, which is a tropical polytope.}\label{fig:log-limit}
\end{figure}

We now introduce the definition of the tropical central path that will be used in the characterization of the log-limit of the central paths $(\cpt)_t$. To this extent, note that any tropically convex set $S$ is closed under the supremum $u \tplus v$ of any of its two points $u, v \in S$. As a consequence, in the case where $S$ is a closed and bounded from above sets,
the supremum of (all the points of) $S$ belongs to $S$. We call this point the \emph{tropical barycenter} of $S$ (by analogy with the barycenter w.r.t.~the uniform measure), and we denote it by $\tbar(S)$. The \emph{tropical central path} (of the parametric family~\eqref{eq:convex_t}) is defined as the function that maps $\lambda \in \R$ to the tropical barycenter of the set $\Ktrop_\lambda \coloneqq \Ktrop \cap \{ u \in \T^n \colon \inpT{\ctrop}{u} \leq -\lambda \}$, where we recall that $\inpT{x}{y} = \bigvee_i (x_i + y_i)$ (\ie, it is the tropical analogue of the inner product of $x$ and $y$). \Cref{fig:cp} provides an illustration.

 % cp
\begin{figure}[t]
\begin{center}
\begin{tikzpicture}[scale=0.9]
\draw[gray!30,very thin] (0.5,0.5) grid (-6.5,-6.5);
\draw[gray!50,->,>=stealth'] (-6.5,0) -- (0.5,0) node[color=black!80,right] {$x_1$};
\draw[gray!50,->,>=stealth'] (0,-6.5) -- (0,0.5) node[color=black!80,above] {$x_2$};

\fill[gray!30] (-6.5,0) -- (0,0) -- (0,-1) -- (-1,-2) -- (-1,-3) -- (-4.5,-6.5) -- (-6.5,-6.5);
\draw[black,very thick] (-6.5,0) -- (0,0) -- (0,-1) -- (-1,-2) -- (-1,-3) -- (-4.5,-6.5);
\draw[dashed,orange, line width=3pt] (0,0) -- (0,-1) -- (-1,-2) -- (-1,-3) -- (-4.5,-6.5);
\end{tikzpicture}
\end{center}	
\caption{The tropical central path of~$\cex_2(t)$. The log-limit of the feasible is represented in gray, and the tropical central path is the dotted piecewise linear curve in orange.}\label{fig:cp}
\end{figure}
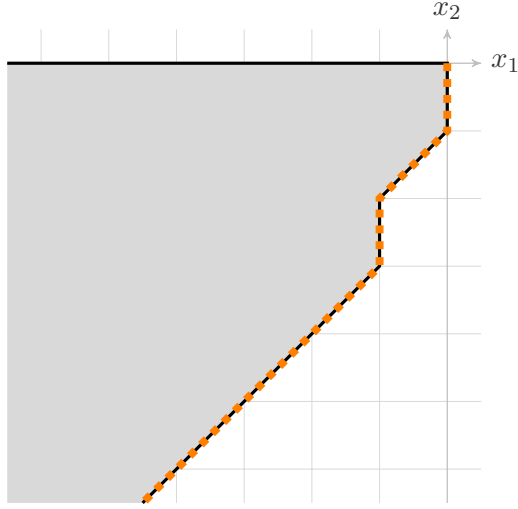

The tropical central path satisfies the following monotonicity and Lipschitzianity properties that will be useful in the proofs:
\begin{lemma}
\label{lemma:cptrop_monoticity}
Let $\lambda' \leq \lambda$, then
\[
\cptrop(\lambda) \leq \cptrop(\lambda') \leq \cptrop(\lambda) + (\lambda - \lambda') \onevector \, .
\]
\end{lemma}

\begin{proof}
Since $\inpT{c}{\cptrop(\lambda)} \leq -\lambda \leq -\lambda'$, then $\cptrop(\lambda) \in \Ktrop_{\lambda'}$, and by definition of the tropical barycenter, $\cptrop(\lambda) \leq \cptrop(\lambda')$. Furthermore, $-\infty \in \Ktrop$ (since $0 \in \Kt$ for all $t > 1$) and $\Ktrop$ is tropically convex by \cref{prop:trop_conv}. Then $\cptrop(\lambda') + (\lambda'-\lambda) \onevector \in \Ktrop$. By linearity, 
\[
\inpT{c}{\cptrop(\lambda') + (\lambda'-\lambda) \onevector} = \inpT{c}{\cptrop(\lambda')} + (\lambda' - \lambda) \leq -\lambda \, .
\]
Therefore $\cptrop(\lambda') + (\lambda'-\lambda) \onevector \leq \cptrop(\lambda)$.
\end{proof}

We can now show that the logarithmic deformation of the central path converges uniformly to the tropical central path, and provide an estimate of the distance between them. Following Assumption~\ref{assmpt:assmpt_K_c}, we define $\delta_\K(t) \coloneqq \distInf(\log_t \Kt;\Ktrop)$ and $\delta_c(t) \coloneqq \distInf(\log_t c_t;\ctrop)$, and we introduce $\delta^* > 0$ such that for all $t > 1$, 
\[
\delta^* \geq \max(\delta_\K(t) , \delta_c(t)) \log t \, .
\]
Moreover, we will make repeated use of the following inequality for $x, y \in \R_{>0}^n$,
\begin{equation}
	\label{eq:log_ineq}
	\tag{\textsc{Log-Ineq}}
	\inpT{\log_t x}{\log_t y} \leq \log_t \big(\inp{x}{y}\big) \leq \inpT{\log_t x}{\log_t y} + \log_t n \, .
\end{equation}

\begin{theorem}[Refinement of \Cref{thm:unif_conv_cp_intro}]
  \label{thm:unif_conv_cp}
 The family of functions $(\lambda \mapsto \log_t \cpt(t^\lambda))_{t}$ converges uniformly to the map $\lambda \mapsto \cptrop(\lambda)$ as $t\to\infty$.
 More precisely, there exists constants $\Gamma>0$ and $t_0>0$ such that for
 all $t\geq t_0$ and for all $\lambda \in \R$,
\begin{align}
  \distInf(\log_t \cpt(t^\lambda), \cptrop(\lambda) )
  \leq \frac{\Gamma}{\log t}
  \, .\label{e:cvg-unif}
\end{align}
Moreover, the constant $\Gamma$
depends only on $\thet$, $n$, $\delta^*$ and on $\Ktrop$.
\end{theorem}
More precisely, the constant $\Gamma$ can be obtained from
any representation of $\Ktrop$ as a finite union of polyhedra,
it has an explicit dependence in the collection of normals
to the facets of these polyhedra.

The idea of the proof is the following. Let $\lambda \in \R$. We suppose, by contradiction, that $\log_t \cpt(t^\lambda)$ is far from $\cpt(\lambda)$, and we build a point $x \in \K_t$ which has a smaller value w.r.t.~the objective function of~\eqref{eq:param_interior_point_program}. To do this, we use the log-like inequalities established in \cref{sec:log_like_barrier}. In particular, in order to apply \Cref{thm:hessian_ineq_ub}, we need to choose $x$ that is sufficiently far away the boundary of $\K_t$. But we also need $\log_t x$ to be close to the point $\cptrop(\lambda)$ of the central path. One difficulty is that, as a tropical barycenter, $\cptrop(\lambda)$ can be in the boundary of the log-limit $\Ktrop$. The careful construction of this point is the purpose of the following lemma, where $\alpha$ is a constant that only depends of $\thet$, $n$, $\delta^*$ and the constant $C_{\Ktrop}$ of \Cref{lemma:perturbation_lemma} in Appendix~\ref{app:proofs} (that only depends on $\Ktrop$). The proof of the following lemma also lies in Appendix~\ref{app:proofs}.

\begin{lemma}\label{lemma:good_lift}
There exists $t_1 > 1$ such that for all $t \geq t_1$ and $\eta > 0$, there is a point $x \in \Kt$ satisfying the following conditions:
\begin{enumerate}
\item $\inp{c_t}{x} \leq \frac{1}{\eta}$,
\item if $y \in \Kt$ is such that $\inp{c_t}{y} \leq \frac{1}{\eta}$, then $y \leq \alpha x$,
\item $x + (K+1) x_i e_i \in \Kt$ and $x - \frac{1}{2} x_i e_i \in \Kt$.
\end{enumerate}
\end{lemma}

\begin{proof}[Proof ({\cref{thm:unif_conv_cp}})]
We fix $t \geq t_1$ and $\eta > 0$, and we consider the point $x$ provided by \Cref{lemma:good_lift}. We define $f_{\eta,t}(z) \coloneqq \eta \inp{c_t}{z} + f_t(z)$, and we denote by $g_t$ the gradient of $f_t$. By \Cref{cor:log_bound}, we have
\begin{align*}
f_{\eta,t}(\cpt(\eta)) - f_{\eta,t}(x) &\geq \eta \big( \inp{c_t}{\cpt(\eta)} - \inp{c_t}{x} \big) + \inp{g_t(x)}{\cpt(\eta)-x} + \frac{1}{n} \inp[\Big]{\frac{1}{x}}{\cpt(\eta)-x} \\
& \qquad + \frac{1}{n} \sum_i \big( \log x_i - \log {\cpt(\eta)}_i \big) \, .
\end{align*}
Observe that $\inp{c_t}{\cpt(\eta)} \geq 0$ (since $c_t \in \R_{\geq 0}^n$ and $\K_t \subset \R_{\geq 0}^n$), and $\eta \inp{c_t}{x} \leq 1$ by hypothesis on $x$. We also have $\inp{\frac{1}{x}}{\cpt(\eta)} \geq 0$ as both vectors have positive components, and $\inp{\frac{1}{x}}{x} = n$. Furthermore, by Cauchy-Schwarz inequality (w.r.t.~$\inp{\cdot}{\cdot}_x$), we have:
\[
\abs{\inp{g_t(x)}{\cpt(\eta) - x}}
\leq \norm{H_t(x)^{-1} g_t(x)}_x \norm{\cpt(\eta)-x}_x \leq \sqrt{\thet} \norm{\cpt(\eta)-x}_x \, ,
\]
where the last inequality follows from the fact that $f_t$ is $\vartheta$-self-concordant. Thanks to \Cref{lemma:good_lift} and \Cref{thm:hessian_ineq_ub}, we know that the Hessian $H_t(x)$ of $f_t$ at $x$ satisfies $H_t(x) \preccurlyeq 4nK^2 \Diag(\frac{1}{x^2})$. Besides, by \cref{lemma:cp_val_ineq} (Appendix~\ref{app:proofs}), we have $\frac{1}{\thet} \inp{c_t}{\cpt(\eta)} \leqslant \frac{1}{\eta}$. Therefore, $\cpt(\eta) \leqslant (\thet \alpha) x$ by hypothesis on $x$, and denoting
\begin{align}
  \beta \coloneqq \max(1,{(1-\thet \alpha)}^2)\enspace,\label{eq:def-beta}
  \end{align}
  we get
\[
\norm{\cpt(\eta)-x}^2_x \leqslant 4nK^2\sum_{i=1}^n {\bigg( 1 - \frac{{\cpt(\eta)}_i}{x_i} \bigg)}^2 \leqslant 4n^2K^2 \beta  \, .
\]
Altogether, we get
\begin{equation}
\label{eq:unif_conv_cp_1}
f_{\eta,t}(\cpt(\eta)) - f_{\eta,t}(x) \geq
-2-2\sqrt{\thet \beta} nK
+ \frac{1}{n} \sum_i \big( \log x_i - \log {\cpt(\eta)}_i \big) \, .
\end{equation}
Let $\eta_t:=t^\lambda$.
We claim that there are constants $M>0$ and $t_0>t_1$ such that 
\begin{equation}
\label{eq:unif_conv_cp_3}
\sum_i \log_t {\cpt(\eta_t)}_i \geq \sum_i {\cptrop(\lambda)}_i - \frac{M}{\log t} \, ,
\end{equation}
for all $\lambda \in \R$ and $t\geq t_0$.
Indeed, suppose that the inequality in~\eqref{eq:unif_conv_cp_3}
does not hold for some $t$.
Then, \eqref{eq:unif_conv_cp_1} yields
\begin{equation}
\label{eq:unif_conv_cp_2}
f_{\eta_t,t}(\cpt(\eta_t)) - f_{\eta_t,t}(x) > -2-2\sqrt{\thet \beta} nK + \frac{1}{n} \sum_i \big( \log_t x_i - {\cptrop(\lambda)}_i \big) \log t + M \, .
\end{equation}
Consider $z \in \Kt$ such that $\distInf(\log_t z; \cptrop(\lambda)) \leqslant \delta_\K(t)$. Then
\begin{align*}
\log_t \inp{c_t}{z} &\leq \log_t n + \delta_\K(t) + \delta_c(t) + \inpT{c}{\cptrop(\lambda)} && \text{using the relation~(\ref{eq:log_ineq}), page~\pageref{eq:log_ineq}}\\
&\leq \log_t n + \delta_\K(t) + \delta_c(t) - \lambda  \, ,
\end{align*}
which implies that $\inp{c_t}{z} \leq n \exp(2\delta^*)/\eta_t$.
We set $\gamma \coloneqq \max(1,n \exp(2\delta^*))$. Since $0 \in \Kt$, the point $ z' \coloneqq \frac{1}{\gamma} z = (1 - \frac{1}{\gamma}) 0 + \frac{1}{\gamma} z$ belongs to $\Kt$. Therefore, by hypothesis on $x$, we have $z' \leq \alpha x$, and subsequently, $z \leq \alpha \gamma x$. Using the fact that $\cptrop(\lambda) \leq \delta_\K(t) + \log_t z$ (by definition of $z$), we obtain
\[
\cptrop(\lambda) \leq  \delta_\K(t) + \log_t (\alpha \gamma) + \log_t x
\, .
\]
Using this inequality in~\eqref{eq:unif_conv_cp_2} yields
\[
f_{\eta_t,t}(\cpt(\eta_t)) - f_{\eta_t,t}(x) > -2-2\sqrt{\thet \beta} nK  - (\delta^* + \log_t (\alpha \gamma)) + M \, .
\]
Taking $M:=\delta^* + 3 + 2\sqrt{\thet \beta} nK$, and assuming that $t> \alpha\gamma$,
  we deduce that $f_{\eta_t,t}(\cpt(\eta_t)) - f_{\eta_t,t}(x) > 0$ holds,
  which contradicts the fact that $\cpt(\eta_t)$ minimizes $f_{\eta_t,t}(\cdot)$. This entails that~\eqref{eq:unif_conv_cp_3} holds for this choice
  of $M$, and for all $t$ large enough.

Furthermore, $\inp{c_t}{\cpt(\eta_t)} \leq \frac{\thet}{\eta_t}$ (\cref{lemma:cp_val_ineq}), and by \Cref{assmpt:assmpt_K_c}, there is $u \in \Ktrop$ such that $\distInf(u, \log_t \cpt(\eta_t)) \leq \delta_\K(t)$. Then
\begin{align*}
\inpT{c}{u} & \leq \delta_\K(t) + \delta_c(t) + \log_t \inp{c_t}{\cpt(\eta)} && \text{using}~\eqref{eq:log_ineq} \\
& \leq \delta_\K(t) + \delta_c(t) + \log_t \thet - \lambda \, .
\end{align*}
By definition of the tropical barycenter, we deduce that 
\[
u \leq \cptrop(\lambda - (\delta_\K(t) + \delta_c(t) + \log_t \thet)) \leq \cptrop(\lambda) + (\delta_\K(t) + \delta_c(t) + \log_t \thet)
\]
where the last inequality follows from the 1-Lipschitz
continuity
of the tropical central path $\lambda \mapsto \cptrop(\lambda)$ with respect
to the sup-norm, see~\cref{lemma:cptrop_monoticity}. As $\log_t \cpt(\eta_t) \leq u + \delta_\K(t)$, we get:
\begin{equation}
\label{eq:unif_conv_cp_4}
\log_t \cpt(\eta_t) \leq \cptrop(\lambda) + 2\delta_\K(t) + \delta_c(t) + \log_t \thet \, .
\end{equation}
Applying \cref{lemma:component_ineq} (\Cref{app:proofs}) to \eqref{eq:unif_conv_cp_3} and \eqref{eq:unif_conv_cp_4} yields that:
\begin{align*}
  \distInf(\cpt(\eta_t), \cptrop(\lambda )) & \leq n
  (2\delta_\K(t) + \delta_c(t) + \log_t \thet) + \frac{M}{\log t} \\
  & \leq \Big((3n +1) \delta^* + \log \thet +
   3 + 2\sqrt{\thet \beta} nK \Big) / \log t \, , \label{eq:final_bound} 
\end{align*}
for all $t\geq t_0$ and for all $\lambda\in \R$.
The bound is of the form $\Gamma/\log t$, in which
the constant $\Gamma$ depends only of $n$, $\delta^*$
and $\vartheta$ (recalling that $K=4\vartheta +1$),
and, through $\beta$ and $\alpha$ (see~\eqref{eq:def-beta}),
on the constant $C_{\Ktrop}$ associated
to the semilinear set $\Ktrop$.
\end{proof}

\section{A tropical lower bound on the iteration complexity of interior point methods}\label{sec:complexity}

Interior point methods follow the central path for increasing values of $\eta$ up to finding an approximate solution to the convex program. In doing so, they build a polygonal curve in a certain neighborhood of the central path. In \Cref{subsec:neighborhoods}, we discuss the neighborhoods used in the literature of IPMs, and show that they all fit in a family of multiplicative neighborhoods of the central path. In \Cref{subsec:lower_bound}, we prove that the log-limit of the latter uniformly collapse to the tropical central path. Finally, we show how we deduce from this property a general tropical lower bound on the number of iterations of IPMs in the case of linear programming.

\subsection{Multiplicative neighborhoods of the central path}\label{subsec:neighborhoods}

In this section, we consider a general (\ie, non-parametric) formulation of the barrier problem:
\begin{equation}\label{eq:barrier_pb}
\Minimize \enspace \eta \inp{c}{x} + f(x) \enspace \SubjectTo \enspace x \in D_f
\end{equation}
where $f$ is a $\thet$-self-concordant barrier with domain $D_f \subset \R_{> 0}^n$, and $\eta > 0$. This problem induces a central path that we denote by $\cp$. 

Given $0 < \multlow < 1 < \multupp$, we define the \emph{multiplicative neighborhood (of parameters $\multlow, \multupp$)} of the point of the central path $\cp(\eta)$ as:
\[
\Mcal_{\multlow,\multupp}(\eta) \coloneqq \big\{ x \in D_f \colon \multlow \cp(\eta) \leq x \leq \multupp \cp(\eta) \big\} \, ,
\]

While there is an abundant literature on the algorithmics of IPMs in the special case of the logarithmic barrier (see~\cite{Wright} for an account), the only references in the case of general barriers we are aware of is~\cite{NesterovN94,Renegar01}. The main reason is that computing with barriers other than the logarithmic one usually requires involved iteration schemes, see for instance~\cite{LeeSidford19} for the case of Lewis weights barrier, and~\cite{DeKlerk} for that of the entropic barrier.

Interestingly, Renegar develops in~\cite{Renegar01} several iteration schemes, namely short-step, long-step and predictor-corrector methods. As far as we know, this constitutes the state of the art in the case of general barriers. We propose to rely on his description of these methods in order to discuss the kind of neighborhoods used. 

To this purpose, given $M > 0$, we introduce the following neighborhood of the point $\cp(\eta)$ of the central path ($\eta > 0)$:
\[
\mathcal{N}_M(\eta) \coloneqq \bigg\{ x \in D_f \colon f_\eta(x) - f_\eta(\cp(\eta)) \leq M \bigg\} \, .
\]
where $f_\eta(x) \coloneqq \eta \inp{c}{x} + f(x)$. We claim that this family of neighborhoods encompasses the different neighborhoods of the central path used in the various methods presented in~\cite{Renegar01}. 

Short- and long-step interior point methods consider a neighborhood $\Ncal^{\mathrm{step}}_\rho(\eta)$ of the point $\cp(\eta)$ of the form
\[
\Ncal^{\mathrm{step}}_\rho(\eta) \coloneqq \{ x \in D_f \colon \norm{n_\eta(x)}_x \leq \rho \} \, ,
\]
for some $\rho > 0$, and where $n_\eta(x) \coloneqq -H(x)^{-1} (\eta c + g(x))$ corresponds to the Newton direction used to iterate from the point $z$. The parameter $\rho$ satisfies $\rho \leq 1/4$. The following result is a consequence of~\cite[Theorem~2.2.5]{Renegar01}.
\begin{lemma}
Let $\rho \leq 1/4$. Then for all $\eta > 0$, $\Ncal^{\mathrm{step}}_\rho(\eta) \subset \Ncal_{3/16}(\eta)$.
\end{lemma}

\begin{proof}
Let $x \in \Ncal^{\mathrm{step}}_\rho(\eta)$. We apply~\cite[Theorem~2.2.5]{Renegar01} to the map $f_\eta$, which yields $\norm{\cp(\eta) - x}_x  \leq \frac{3}{4}$. By convexity of $f$, we have
\[
f_\eta(x) - f_\eta(\cp(\eta)) \leq -\inp{\eta c + g(x)}{\cp(\eta) - x} = \inp{n_\eta(x)}{\cp(\eta) - x}_x \leq \norm{n_\eta(x)}_x \norm{\cp(\eta) - x}_x \leq \frac{3}{16} \, . \qedhere
\]
\end{proof}

The predictor-corrector method alternates between a predictor step and a corrector step. The former starts from a point $x = x^1$ such that $\norm{n_{\vert L(x)}(x)}_{x} \leq \frac{1}{13}$, where $L(x)$ is the affine subspace of the points $y$ such that $\inp{c}{y} = \inp{c}{x}$, and $n_{\vert L(x)}(x)$ is the orthogonal projection of the Newton step $-H(x)^{-1} g(x)$ of $f$ at $x$ onto $L(x)$. It then follows a halfline $x - s c_x$ ($s \geq 0$), where $c_x \coloneqq H(x)^{-1} c$. The vector $-c_x$ corresponds to the so-called ``affine scaling'' direction used in the case of the logarithmic barrier. If $\bar s$ is the supremum of the $s \geq 0$ such that $x - s c_x \in D_f$, the predictor step stops at the point $x' \coloneqq x - \sigma \bar s c_x$, where $\sigma < 1$ is a prescribed constant. In the proof of the polynomial time convergence of the algorithm~\cite[p.~54]{Renegar01}, Renegar shows that 
\[
f(x') - f(\cp(\eta_2)) \leq \thet \log \frac{1}{1 - \sigma} + \frac{1}{154} =: \Mpc \, ,
\]
where $\eta_2 > 0$ satisfies such that $\inp{c}{\cp(\eta_2)} = \inp{c}{x'}$. The algorithm then performs a corrector step in the subspace $L(\cp(\eta_2))$ (using a number of exact line searches) to get back to a point $x^2$ such that $\inp{c}{x^2} = \inp{c}{\cp(\eta_2)}$ such that $\norm{n_{\vert L(x^2)}(x^2)}_{x^2} \leq \frac{1}{13}$. The proof of Renegar actually shows that all points $z$ in the  trajectory followed by the predictor-corrector method (\ie, the sequence of segments between the successive iterate satisfies $f(z) - f(\cp(\eta)) \leq \Mpc$, where $\eta > 0$ is chosen such as $\inp{c}{\cp(\eta)} = \inp{c}{z}$. We obviously have $z \in \Ncal_\Mpc(\eta)$.

We conclude this section by showing that the multiplicative neighborhoods introduced above indeed captures the neighborhoods of the form $\Ncal_M(\cdot)$. The proof relies again on the log-like properties of the barrier $f$:
\begin{proposition}
\label{prop:mult_neigh}
Let $M > 0$. There exist $\multlow, \multupp > 0$ depending only on $n$ and $M$ such that for all $\eta > 0$, $\Ncal_M(\eta) \subset \Mcal_{\multlow, \multupp}(\eta)$. 
\end{proposition}

\begin{proof}
Consider $\eta>0$ and $x \in \mathcal{N}_M(\eta)$, meaning that $f_\eta(x) \leq M + f_\eta(\cp(\eta))$. By \cref{cor:log_bound}, we have
\begin{align}
M \geq f_\eta(x) - f_\eta(\cp(\eta)) & \geq \eta \inp{c}{x - \cp(\eta)} +  \inp{g(\cp(\eta))}{x-\cp(\eta)} + \inp{\frac{1}{n\cp(\eta)}}{x} - \inp{\frac{1}{n\cp(\eta)}}{\cp(\eta)} \nonumber\\ 
&\qquad + \frac{1}{n} \sum_i \big( \log {\cp(\eta)}_i - \log x_i \big) \, .\label{eq:chain-temp}
\end{align}
Since $\cp(\eta)$ is a point of the central path, then $g(\cp(\eta)) = - \eta c$. Furthermore $\inp{\frac{1}{n\cp(\eta)}}{\cp(\eta)} = 1$. Denoting $\phi: z \mapsto z - \log(z)-1$, we rewrite the inequality~\eqref{eq:chain-temp} as
\begin{equation}
\label{eq:mult_neigh_1}
M  \geq \frac{1}{n} \sum_i \phi \big( \frac{x_i}{{\cp(\eta)}_i} \big) \, .
\end{equation}
Since the function $\phi$ is bounded from below by $0$, then for all $i$, $\phi \big( \frac{x_i}{{\cp(\eta)}_i} \big) \leq nM$. Since $\phi$ is convex on $\R_{>0}$ and tends to $\infty$ at the points $0$ and $+\infty$,  $\phi^{-1}([0,nM)])$ is an interval of $\R_{>0}$, hence there are two constants $\multlow \leq \multupp$ depending only on $n$ and $M$ such that for all $i \in [n]$,
\[
\multlow \leq \frac{x_i}{{\cp(\eta)}_i} \leq \multupp \, . \qedhere
\]
\end{proof}

\subsection{The lower bound}
\label{subsec:lower_bound}

One of the benefits of introducing the multiplicative neighborhoods of the central path is to provide an elementary proof that their log-limit is reduced to the tropical central path. To this purpose, we return to the parametric formulation of the barrier problem described in Section~\ref{sec:uniform}, and we add explicitly the dependency in $t$ of the neighborhood, \ie, 
\[
\Mcal_{\multlow,\multupp,t}(\eta) \coloneqq \big\{ x \in \Kt \colon \multlow \cpt(\eta) \leq x \leq \multupp \cpt(\eta) \big\}  \, .
\]
where we recall that $f_{\eta,t}(z) = \eta \inp{c_t}{z} + f_t(z)$. We also define
\[
\Mcal_{\multlow,\multupp,t}([\ubar \eta, \bar \eta]) \coloneqq \bigcup_{\ubar \eta \leq \eta \leq \bar \eta} \Mcal_{\multlow,\multupp,t}(\eta) 
\]
the induced neighborhood of the portion of central path $\cpt([\ubar \eta, \bar \eta])$.

\begin{theorem}
\label{thm:unif_conv_neighborhood}
Let $\multlow, \multupp > 0$. There exists a constant
$\Gamma_{\multlow,\multupp}>0$ and $t_0>0$ such that,
for all $t\geq t_0$ and for all $\lambda\in \R$,
\[
\distInf(x, \cptrop(\lambda) ) \leq \frac{\Gamma_{\multlow,\multupp}}{\log t}
\enspace,\text{ for all }x \in \Mcal_{\multlow,\multupp,t}(t^\lambda) 
\enspace .
\]
\end{theorem}

\begin{proof}
Let $\Gamma$ and $t_0$ be as in \cref{thm:unif_conv_cp}.
Then, for all $t\geq t_0$ and for all $\lambda\in\R$, we have that 
$\distInf(\log_t \cpt(t^\lambda);\cptrop(\lambda)) \leq \Gamma/(\log t)$.
Besides, for all $x \in \Mcal_{\multlow,\multupp,t}(t^\lambda)$, we have $\distInf(\log_t x, \log_t \cpt(t^\lambda)) \leq \max(\log_t \multlow, -\log_t \multupp)$.
 Therefore, for $t\geq t_0$, 
\[
\distInf(x, \cptrop(\lambda)) \leq \frac{\Gamma}{\log t}
+ \max(\log_t \multlow, -\log_t \multupp) = \frac{\Gamma_{\multlow,\multupp}}{\log t}
\enspace,
\]
with $\Gamma_{\multlow,\multupp}:=\Gamma+ \max(\log \multlow, -\log \multupp)$.
\end{proof}

Our aim is now to establish the iteration complexity lower bound in terms of the ``complexity'' of the tropical central path. We establish the lower bound in the case where the convex sets $\Kt$ are polyhedra, \ie, we focus on the complexity of linear programming. Recall that~\Cref{assmpt:assmpt_K_c} is always satisfied in this setting under a genericity condition.

As shown in~\cite{ABGJ18} (Lemma~5 and Prop.~16), in the case of linear programming, the tropical central path is the concatenation of finitely many tropical segments. The number of these segments acts as the complexity measure of the tropical central path. More precisely, given a section $\cptrop([\ubar \lambda, \bar \lambda])$ of the tropical central path, we denote by $\gamma([\ubar \lambda, \bar \lambda])$ the minimal number of tropical segments needed to describe it. We denote by $B_\infty(x;\eps)$ the $d_\infty$-ball of center $x$ and radius $\eps$, and we introduce the following tubular neighborhood or radius $\eps > 0$ of the section $\cptrop([\ubar \lambda, \bar \lambda])$ of the tropical central path:
\[
\Tcal([\ubar \lambda,\bar \lambda];\eps) \coloneqq \bigcup_{\ubar \lambda \leq \lambda \leq \bar \lambda} B_\infty(\cptrop(\lambda);\eps)
\]
where $\eps > 0$. The main result that we use from~\cite{ABGJ18} is the following.
\begin{proposition}[{restatement of \cite[Prop.~28]{ABGJ18}}]\label{prop:tube}
If $\eps > 0$ is small enough, then any concatenated sequence of  tropical segments included in the neighborhood $\Tcal([\ubar \lambda,\bar \lambda];\eps)$ that starts in $B_\infty(\cptrop(\ubar \lambda);\eps)$ and ends in $B_\infty(\cptrop(\bar \lambda);\eps)$ requires at least $\gamma([\ubar \lambda, \bar \lambda])$ tropical segments.
\end{proposition}

In the following statement, the sequence of segments corresponds to the trajectory followed by the interior point method. In light of the discussion of~\Cref{subsec:neighborhoods}, we assume that it is contained in a multiplicative neighborhood of the central path.
\begin{theorem}\label{th:complexity}
Let $\ubar \lambda < \bar \lambda$. Provided that $t > 1$ is large enough, any sequence of segments $[x^0, x^1] \cup [x^1, x^2] \cup \dots \cup [x^{p-1}, x^p]$ contained in the neighborhood $\Mcal_{\multlow,\multupp,t}([t^{\ubar \lambda},t^{\bar \lambda}])$ of $\cpt$ and such that $x^0 \in \Mcal_{\multlow,\multupp,t}(t^{\ubar \lambda})$ and $x^p \in \Mcal_{\multlow,\multupp,t}(t^{\bar \lambda})$ contains at least $\gamma([\ubar \lambda, \bar \lambda])$ segments.
\end{theorem}

\begin{proof}
We consider the associated sequence of tropical segments $\tsegm(\log_t x^i, \log_t x^{i+1})$ ($i \in [p]$). As shown in~\cite[Lemma~8]{ABGJ18}, we have $
d_\infty\big(\log_t [x^i, x^{i+1}], \tsegm(\log_t x^i, \log_t x^{i+1})\big) \leq \log_t 2$. Thus, taking $t$ large enough, we can suppose that:
\begin{equation}\label{eq:d_segm}
\text{for all} \enspace i \in [p] \, , \enspace d_\infty\big(\log_t [x^i, x^{i+1}], \tsegm(\log_t x^i, \log_t x^{i+1})\big) \leq \frac{\eps}{2} \, ,
\end{equation}
and, by \Cref{thm:unif_conv_neighborhood}, for all $\lambda\in\R$,
\begin{equation}\label{eq:d_neigh}
  \distInf(x, \cptrop(\lambda)) \leq \frac{\eps}{2} \, ,
  \text{ for all } x \in \Mcal_{\multlow,\multupp,t}(t^\lambda)
  \enspace,
\end{equation}
where $\eps > 0$ is chosen as in \Cref{prop:tube}. 

We claim that the sequence of tropical segments $\tsegm(\log_t x^i, \log_t x^{i+1})$ ($i \in [p]$) is included in the neighbhorhood $\Tcal([\ubar \lambda,\bar \lambda];\eps)$ of the tropical central path. Indeed, take $x \in \tsegm(\log_t x^i, \log_t x^{i+1})$. By~\eqref{eq:d_segm}, there exists $x' \in [x^i, x^{i+1}]$ such that $d_\infty(\log_t x,\log_t x') \leq \eps / 2$. Besides, since we have $[x^i, x^{i+1}] \subset \Mcal_{\multlow,\multupp,t}([t^{\ubar \lambda},t^{\bar \lambda}])$, there exists $\eta =t^\lambda \in [t^{\ubar \lambda}, t^{\bar \lambda}]$ such that $x' \in \Mcal_{\multlow,\multupp,t}(\eta)$. Using~\eqref{eq:d_neigh}, we get: 
\[
d_\infty(\log_t x,\cptrop(\lambda)) \leq
d_\infty(\log_t x,\log_t x')+
d_\infty(\log_t x',\cptrop(\lambda)) \leq \eps \, ,
\] 
showing the claim.
By the same argument, we have $\log_t x_0 \in B_\infty(\cptrop(\ubar \lambda);\eps)$ and $\log_t x^p \in B_\infty(\cptrop(\lambda);\eps)$. 

We conclude by applying \Cref{prop:tube}.
\end{proof}

\begin{remark}
  It is possible to give an explicit a lower bound for the value of $t$ in \Cref{th:complexity}, in terms of the constant $\Gamma$ of Theorem~\ref{thm:unif_conv_cp}, $\multlow$, $\multupp$ and the quantity $\eps$ of \Cref{prop:tube}. This would lead to a doubly exponential lower bound, like in the case of the logarithmic barrier (see~\cite{ABGJ18}).  See also~\cite[Fig.~2]{sirev} for a numerical illustration of the latter results, showing that in this case, and for doubly exponential values of $t$, the $\log_t$ image 
  of the trajectory followed by the predictor-corrector method
  effectively stays in a small neighborgood of the tropical central  path.
\end{remark}

The purpose of IPMs is to get an $\eps$-approximation of the optimal value. Therefore, it is relevant to establish a lower bound on the complexity expressed in terms of the value of the objective function rather than of the parameter $\eta$. In general, we know that $\inp{c}{\cp(\eta)} \leq \frac{\thet}{\eta}$ for all $\eta > 0$ (see \cref{lemma:cp_val_ineq}). The following result shows that, in the case of a parametric family of convex programs, the value of the objective function on the central path is actually in $\Omega(\frac{1}{\eta})$, under some mild assumptions. We set $\lambda^* \coloneqq -\sup_{x \in \Ktrop} \inpT{\ctrop}{x} \in \T$. 
\begin{proposition}
\label{prop:value_central_path}
For all $t \geq t_0$, for all $\eta > 0$ such that $\log_t \eta \geq \lambda^*$, we have:
\[
\frac{1}{\eta \exp(\Gamma + \delta^*)} \leq \inp{c_t}{\cpt(\eta)} \leq \frac{\vartheta}{\eta} \, .
\]
\end{proposition}

\begin{proof}
The second inequality is precisely \cref{lemma:cp_val_ineq}. We focus on the first inequality.

We fix $t \geq t_0$ and $\eta > 0$. Let $\lambda \coloneqq \log_t \eta \geq \lambda^*$. We claim that $\inpT{\ctrop}{\cptrop(\lambda)} = -\lambda$. By definition of $\cptrop(\lambda)$, we know that $\inpT{\ctrop}{\cptrop(\lambda)} \leq -\lambda$. Let $\lambda' \coloneqq - \inpT{\ctrop}{\cptrop(\lambda)} \geq \lambda$. We introduce a point $x \in \Ktrop$ such that $\inpT{\ctrop}{x} \geq -\lambda$ (which is possible because $\Ktrop$ is closed, as a regular set). Ley $\nu \coloneqq (-\frac{\lambda + \lambda^\prime}{2} - \inpT{\ctrop}{x})$. Observe that $\nu \leq 0$. 
Therefore, the point $y \coloneqq \cptrop(\lambda) \vee (\nu + x)$ belongs to the set $\Ktrop$ (by tropical convexity). Besides, $\inpT{\ctrop}{y} = (-\lambda') \vee (-\frac{\lambda + \lambda'}{2}) = -\frac{\lambda + \lambda'}{2} \leq -\lambda$. We deduce that $y \in \Ktrop_\lambda$, so that $y \leq \cptrop(\lambda)$. However, by definition of $y$, we have $y \geq \cptrop(\lambda)$, hence $y = \cptrop(\lambda)$. We deduce that
\[
-\frac{\lambda + \lambda'}{2} = \inpT{\ctrop}{y} = \inpT{\ctrop}{\cptrop(\lambda)} = -\lambda' \, ,
\]
which implies $\lambda = \lambda'$.

Furthermore, by \cref{thm:unif_conv_cp}, we have $
\distInf(\log_t \cpt(\eta), \cptrop(\lambda)) \leq \frac{\Gamma}{\log(t)}$.
Combining this with the inequality~\eqref{eq:log_ineq} and the fact that $\distInf(\log_t c_t, \ctrop) \leq \frac{\delta^*}{\log t}$ yields
\[
\log_t \inp{c_t}{\cpt(\eta)} \geq \inpT{\ctrop}{\cptrop(\lambda)} - \frac{\Gamma + \delta^*}{\log t} = -\log_t \eta - \frac{\Gamma + \delta^*}{\log t} \, ,
\]
which implies $\inp{c_t}{\cpt(\eta)} \geq \frac{1}{\eta \exp(\Gamma+\delta^*)}$. This bound is valid for any $\eta > 0$ and $t \geq t_0$.
\end{proof}

\begin{lemma}\label{lemma:eta_bounds}
Let $t \geq t_0$. Suppose that $x \in \Mcal_{\multlow,\multupp,t}(\eta)$ for some $\eta \geq t^{\lambda^*}$. We have $
\frac{\multlow}{v \exp(\Gamma+\delta^*)} \leq \eta \leq \frac{\multupp \thet}{v}$, where $v = \inp{c_t}{x}$.
\end{lemma}

\begin{proof}[\proofname{} of \Cref{lemma:eta_bounds}]
Using the definition of $\Mcal_{\multlow,\multupp,t}(\eta)$ and \Cref{prop:value_central_path}, we have
\[
\frac{\multlow}{\eta \exp(\Gamma+\delta^*)} \leq v \leq \frac{\multupp \thet}{\eta} \, . \qedhere
\]
\end{proof}

\begin{theorem}\label{th:complexity2}
Consider a sequence of segments $[x^0, x^1] \cup [x^1, x^2] \cup \dots \cup [x^{p-1}, x^p]$ contained in the neighborhood $\Mcal_{\multlow,\multupp,t}([t^{\lambda^*}, +\infty[)$ of $\cpt$, and such that $\ubar v \coloneqq \inp{c_t}{x^0} \geq \inp{c_t}{x^p} =: \bar v$. 

Provided that $t$ is large enough, this sequence contains at least $\gamma([-\log_t \ubar v, -\log_t \bar v])$ segments.
\end{theorem}

\begin{proof}
We first deal with the case where $\ubar v \geq \inp{c_t}{x^i} \geq \bar v$ for all $i \in [p]$.

Take $t \geq t_0$. We claim that the sequence of segments is included in $\Mcal_{\multlow,\multupp,t}\big([\frac{\multlow}{\bar v \exp(\Gamma+\delta^*)}, \frac{\multupp \thet}{\ubar v}]\big)$. Indeed, let $x \in [x^i, x^{i+1}]$, and $\eta \geq t^{\lambda^*}$ such that $x \in \Mcal_{\multlow,\multupp,t}(\eta)$. Using \Cref{lemma:eta_bounds} and $\inp{c_t}{x} \in [\ubar v, \bar v]$, we get that $\eta \in [\frac{\multlow}{\bar v \exp(\Gamma+\delta^*)}, \frac{\multupp \thet}{\ubar v}]$. 

We apply the same arguments as in the proof of \Cref{th:complexity}, and consider $t$ sufficiently large so that~\eqref{eq:d_segm} and~\eqref{eq:d_neigh} hold. Setting $\ubar \lambda \coloneqq -\log_t \ubar v + \log_t \frac{\multlow}{\exp(\Gamma+\delta^*)}$ and $\bar \lambda \coloneqq -\log_t \bar v + \log_t \multupp \thet$, we obtain that the sequence of tropical segments $\tsegm(\log_t x^i, \log_t x^{i+1})$ is contained in $\Tcal([\ubar \lambda,\bar \lambda];\eps)$.

Moreover, let $\eta^0 \geq t^{\lambda^*}$ such that $x^0 \in \Mcal_{\multlow,\multupp,t}(\eta^0)$. By~\eqref{eq:d_neigh}, we have $d_\infty(\log_t x^0, \cptrop(-\log_t \eta^0)) \leq \eps / 2$. Besides, by \Cref{lemma:cptrop_monoticity} and \Cref{lemma:eta_bounds}, we have:
\[
d_\infty(\cptrop(-\log_t \eta^0), \cptrop(\ubar \lambda)) \leq \log_t \eta^0 - \ubar \lambda \leq \log_t \multupp \thet \leq \eps / 2\, .
\]
where the last inequality is obtained up to taking $t$ slightly larger. This proves that $\log_t x^0 \in B_\infty(\cptrop(\ubar \lambda); \eps)$. The same kind of argument shows that $\log_t x^p \in B_\infty(\cptrop(\bar \lambda); \eps)$. We deduce from \Cref{prop:tube} that $p \geq \gamma([\ubar \lambda, \bar \lambda])$. 

We note that $\ubar m < 1 < \bar m$, and $\Gamma$ and $\delta^*$ are positive. Without loss of generality, we can also assume that $\thet \geq 1$.\footnote{The constant $\thet$ is an upper bound on the complexity values of the barriers $f_t$.} We deduce that $[-\log_t \ubar v + \log_t \frac{\multlow}{\exp(\Gamma+\delta^*)}, -\log_t \bar v + \log_t \multupp \thet] \supset [-\log_t \ubar v, -\log_t \bar v]$. Since $\gamma(\cdot)$ is monotone, we obtain the expected result.

We finally deal with the general case. We let ${\ubar v}' \coloneqq \max_i \inp{c_t}{x^i} \geq \ubar v$ and ${\bar v}' \coloneqq \min_i \inp{c_t}{x^i} \leq \bar v$. Then, there exists a subsequence $[x^k, x^{k+1}] \cup \dots \cup [x^{l-1}, x^l]$ such that for all $i \in \{k, k+1, \dots, l\}$, ${\ubar v}' \geq \inp{c_t}{x^i} \geq {\bar v}'$. Thus, $p \geq l-k \geq \gamma([-\log_t {\ubar v}', -\log_t {\bar v}'])$ by the first part of the proof. Besides, the latter quantity is lower bounded by $\gamma([-\log_t \ubar v, -\log_t \bar v])$. 
\end{proof}

\section{Application to the linear program~\eqref{eq:cex}}\label{sec:cube}

In this section, we apply the result of Section~\ref{sec:complexity} to the linear program~\eqref{eq:cex}. As shown in Lemma~\ref{lemma:isval} (in Appendix~\ref{app:cube}), the log-limit of the feasible set of~\eqref{eq:cex} is given by the set of points $x \in \T^n$ satisfying the following inequalities
\begin{equation}\label{eq:tcex}
\allowdisplaybreaks
\begin{aligned}
\bigvee_{j = 1}^{i-1} (-u_i + x_j) \vee (-u_{i+1} + 1 + x_i) & \leq \bigvee_{j = i+1}^{n-1} (-u_j + x_j) \vee x_n \vee (-u_n) && \text{for all}\enspace i = 1, \dots, n-1 \\
\bigvee_{j = 1}^n x_j & \leq 0 \\
-\infty \leq x_1 \leq x_2 & \leq \dots \leq x_{n-1} \leq u_n + x_n \, .
\end{aligned} \tag{$\tcex_n$}
\end{equation}
These inequalities defines a tropical polyhedron (even a tropical polytope). Lemma~\ref{lemma:isval} further provides the fact that Assumption~\ref{assmpt:assmpt_K_c} is fulfilled.

The point $\cptrop(\lambda)$ of the tropical central path with parameter $\lambda \in \R$ is therefore given by the tropical barycenter (\ie, the supremum w.r.t.~the componentwise order) of the set of points $x \in \T^n$ satisfying~\eqref{eq:tcex} and $x_n \leq -\lambda$. The following lemma is straightforward:
\begin{lemma}\label{lemma:x_n}
For all $\lambda \in \R$, $\big(\cptrop(\lambda)\big)_n = \min(-\lambda, 0)$. Moreover, $\cptrop(\lambda) = 0$ for all $\lambda \leq 0$.
\end{lemma}

\begin{proof}
For all $x$ satisfying~\eqref{eq:tcex} and $x_n \leq -\lambda$, we have $x_n \leq \min(-\lambda, 0)$. Moreover, the point $\begin{bsmallmatrix} -\infty_{n-1} \\ \min(-\lambda, 0) \end{bsmallmatrix}$ satisfies the constraints of~\eqref{eq:tcex}. This proves that $\big(\cptrop(\lambda)\big)_n = \min(-\lambda, 0)$. 

Now suppose that $\lambda \leq 0$.
Since any point $x$ satisfying~\eqref{eq:tcex} verifies $x \leq 0$ (owing to the inequality at the second line in~\eqref{eq:tcex}) and the point $0$ satisfies~\eqref{eq:tcex} (owing to the presence of $x_n$ at the right-hand side of the first series of inequalities in~\eqref{eq:tcex}, and to the fact that $0\leq u_i$ and $u_{i+1}\geq 1$ for all $i\geq 1$), we deduce that $\cptrop(\lambda) = 0$.
\end{proof}

We now fully describe the tropical central path associated of the linear program~\eqref{eq:cex}. This description is done by induction on $n \geq 1$, which motivates to use the notation $\cptrop_n$ rather than simply $\cptrop$ in order to keep track of the dependency on $n$. We  prove that $\cptrop_n$ is made of two copies of $\cptrop_{n-1}$ (resp.~when $\lambda \leq u_n-1$ and~$\lambda \geq u_n+1$) and an extra tropical segment inbetween. The first copy corresponds to lifting $\cptrop_{n-1}$ to $n$-dimensional space by inserting the constant value $0$ for $x_{n-1}$. The second copy is essentially a shift of $\cptrop_{n-1}$. We refer to Table~\ref{tab:value} for a value table of the tropical central path of~\eqref{eq:cex} in the case where $n = 2, 3, 4$, and to Figure~\ref{fig:trop_cp} for an illustration when $n = 3$. In the following proposition, given a vector $z$ of size $k$ and $I \subset [k]$, the notation $z_I$ stands for the vector $(z_i)_{i \in I}$. 
 
\begin{proposition}\label{prop:trop_cp}
The tropical central path is given by the following relations: for all $\lambda \in \R$,
\begin{align*}
\allowdisplaybreaks
\cptrop_1 (\lambda) & = \min(0,-\lambda) \\	
\text{and for all} \enspace n > 1 \, , \enspace \cptrop_n(\lambda) & = 
\begin{dcases}
0 & \text{if} \enspace \lambda \leq 0 \, , \\
\begin{bsmallmatrix}
(\cptrop_{n-1}(\lambda))_{[n-2]} \\
0 \\
-\lambda
\end{bsmallmatrix} & \text{if} \enspace 0 \leq \lambda \leq u_n - 1 \, , \\[1ex]
\begin{bsmallmatrix} 
-u_{n-1} e \\
0 \\
-(u_n - 1) 
\end{bsmallmatrix}
+
\begin{bsmallmatrix}
\big[((u_n - 1) - \lambda) \vee (-1)\big] e \\
(u_n - 1) -\lambda  
\end{bsmallmatrix}
& \text{if} \enspace u_n - 1 \leq \lambda \leq u_n + 1 \, , \\[1ex]
\begin{bsmallmatrix} 
-(u_{n-1} + 1)e \\
-1 \\
-(u_n + 1) 
\end{bsmallmatrix}
+
\begin{bsmallmatrix} 
\cptrop_{n-1}(\lambda - (u_n + 1)) \\
u_n + 1 - \lambda
\end{bsmallmatrix} & \text{if} \enspace u_n + 1 \leq \lambda \, .
\end{dcases}
\end{align*}
\end{proposition}

\begin{table}
\caption{Value table of the tropical central path of~\eqref{eq:cex} for $0 \leq \lambda \leq 2 u_n$, when $n = 2$~(left), $3$~(right), and $4$~(bottom).}\label{tab:value}
\medskip
\centering
\scriptsize
\begin{tabular}{c@{\quad}r@{\quad}r@{\quad}r@{\quad}r@{\quad}r}
\toprule
$\lambda$ & 0 & 1 & 2 & 3 & 4 \\
\midrule
$x_1$ & 0 & 0 & -1 & -1 & -2 \\
$x_2$ & 0 & -1 & -2 & -3 & -4 \\
\bottomrule
\end{tabular}
\qquad
$%\renewcommand{\arraystretch}{1.75}%
\begin{array}{c@{\quad}r@{\quad}r@{\quad}r@{\quad}r@{\quad}r@{\quad}r@{\quad}r@{\quad}r@{\quad}r@{\quad}r@{\quad}r}
\toprule
\lambda & 0 & 1 & 2 & 3 & 4 & 5 & 6 & 7 & 8 & 9 & 10\\
\midrule
x_1 & 0 & 0 & -1 & -1 & -2 & -3 & -3 & -3 & -4 & -4 & -5 \\
x_2 & 0 & 0 & 0 & 0 & 0 & -1 & -1 & -2 & -3 & -4 & -5 \\
x_3 & 0 & -1 & -2 & -3 & -4 & -5 & -6 & -7 & -8 & -9 & -10 \\
\bottomrule
\end{array}$

\bigskip
\centering
\scriptsize
$%\renewcommand{\arraystretch}{1.75}%
\begin{array}{c@{\quad}r@{\;\;}r@{\;\;}r@{\;\;}r@{\;\;}r@{\;\;}r@{\;\;}r@{\;\;}r@{\;\;}r@{\;\;}r@{\;\;}r@{\;\;}r@{\;\;}r@{\;\;}r@{\;\;}r@{\;\;}r@{\;\;}r@{\;\;}r@{\;\;}r@{\;\;}r@{\;\;}r@{\;\;}r@{\;\;}r}
\toprule
\lambda & 0 & 1 & 2 & 3 & 4 & 5 & 6 & 7 & 8 & 9 & 10 & 11 & 12 & 13 & 14 & 15 & 16 & 17 & 18 & 19 & 20 & 21 & 22\\
\midrule
x_1 & 0 & 0 & -1 & -1 & -2 & -3 & -3 & -3 & -4 & -4 & -5 & -6 & -6 & -6 & -7 & -7 & -8 & -9 & -9 & -9 & -10 & -10 & -11 \\
x_2 & 0 & 0 & 0 & 0 & 0 & -1 & -1 & -2 & -3 & -4 & -5 & -6 & -6 & -6 & -6 & -6 & -6 & -7 & -7 & -9 & -9 & -10 & -11 \\
x_3 & 0 & 0 & 0 & 0 & 0 & 0 & 0 & 0 & 0 & 0 & 0 & -1 & -1 & -2 & -3 & -4 & -5 & -6 & -7 & -8 & -9 & -10 & -11 \\
x_4 & 0 & -1 & -2 & -3 & -4 & -5 & -6 & -7 & -8 & -9 & -10 & -11 & -12 & -13 & -14 & -15 & -16 & -17 & -18 & -19 & -20 & -21 & -22\\
\bottomrule
\end{array}$

\end{table}

\begin{theorem}\label{th:cex_gamma}
  The following equalities hold:
  \begin{align*}
    \gamma([0,u_n-1]) &= 2^{n-1} - 1\enspace ,\\
\gamma([0,2u_n-1]) &= 2^n - 2\enspace,\\
\gamma([0,2u_n]) &= \gamma([0,+\infty[) = 2^n - 1 \enspace .
    \end{align*}
\end{theorem}

\begin{proof}
We show the statement by induction on $n \geq 1$. To this purpose, we denote by $\gamma_n(\cdot)$ the quantity $\gamma(\cdot)$ associated with the tropical central path $\cptrop_n$ of~\eqref{eq:cex}. 

We note that the relations above are trivially verified for $n = 1$. We now suppose that they holds for $n-1$, where $n > 1$. 

By \Cref{lemma:x_n} and \Cref{prop:trop_cp}, we now that $\cptrop_n([0,u_n - 1])$ is precisely the image of $\cptrop_{n-1}([0,u_n - 1])$ by the map $\phi \colon z \mapsto \begin{bsmallmatrix} z_{[n-2]} \\ 0 \\ z_{n-1} \end{bsmallmatrix}$. Besides, given $z, z' \in \T^{n-1}$, we have $\phi(\tsegm(z, z')) = \tsegm(\phi(z), \phi(z'))$. We deduce that $\gamma_n([0,u_n-1]) = \gamma_{n-1}([0,u_n - 1])$. Since $u_n - 1 = 2 u_{n-1}$, we deduce by the induction hypothesis that $\gamma_n([0,u_n-1]) = 2^{n-1} - 1$.

Similarly, if $\lambda \geq u_n + 1$, then $\cptrop_n([u_n+1,\lambda])$ is precisely the image of $\cptrop_{n-1}([0,\lambda - (u_n + 1)])$ by the map $\psi : z \mapsto \begin{bmatrix} 
-(u_{n-1} + 1)e \\
-1 \\
-(u_n + 1) 
\end{bmatrix}
+
\begin{bmatrix} 
z \\
z_{n-1}
\end{bmatrix}$, thanks to Lemma~\ref{lemma:x_n} and \Cref{prop:trop_cp}. Since $\psi(\tsegm(z, z')) = \tsegm(\psi(z), \psi(z'))$ or all $z, z' \in \T^{n-1}$, we have $\gamma_n([u_n+1, \lambda]) = \gamma_{n-1}([0, \lambda - (u_n + 1)])$. 

Finally, $\cptrop_n([u_n - 1, u_n])$ and $\cptrop_n([u_n, u_n+1])$ are two ordinary segments directed along $-e^{[n]}$ and $-e^n$. As a consequence of \cite[Lemma 5]{ABGJ18}, $\cptrop_n([u_n - 1, u_n + 1])$ is made of one tropical segment.

Note that the last ordinary segment in $\cptrop_n([0,u_n - 1])$ is directed by a vector of the form $-e^K$ where $K \subset [n]$ and $n-1 \notin K$. Since the first ordinary segment of $\cptrop_n([u_n - 1, u_n + 1])$ is directed by $-e^{[n]}$, and $[n] \subset K$, we deduce from \cite[Lemma 5]{ABGJ18} that $\gamma_n([0,u_n + 1]) = \gamma_n([0,u_n-1]) + 1$. Similarly, if $\lambda > u_n + 1$, the first ordinary segment in $\cptrop_n([u_n+1,\lambda])$ is directed along a vector of the form $-e^L$ where $L \subset [n]$ contains $n-1$ (indeed, $(\cptrop_n(\lambda'))_{n-1} = u_n + 1 - \lambda'$ for all $u_n + 1 < \lambda' \leq \lambda$ by \Cref{lemma:x_n} and \Cref{prop:trop_cp}). Since the last ordinary segment of $\cptrop_n([0,u_n + 1])$ is directed along $-e^n$, we conclude that $
\gamma_n([0, \lambda]) = \gamma_n([0,u_n-1]) + 1 + \gamma_n([u_n+1, \lambda])$. Equivalently, 
\[
\gamma_n([0, \lambda]) = 2^{n-1} + \gamma_{n-1}([0, \lambda - (u_n + 1)]) \, .
\]
Thus, $\gamma_n([0,2u_n - 1]) = 2^{n-1} + \gamma_{n-1}([0, u_n - 2]) = 2^{n} - 2$, thanks to the induction hypothesis and the relation $u_n = 2u_{n-1} + 1$. Similarly, $\gamma_n([0, 2 u_n]) = 2^{n-1} + \gamma_{n-1}([0, u_n - 1]) = 2^n - 1$, and $\gamma_n([0, +\infty[) = 2^{n-1} + \gamma_{n-1}([0, +\infty[) = 2^n - 1$.
\end{proof}

We are now ready to prove \Cref{th:cex_intro}, that we restate in a more precise way as follows:
\begin{theorem}\label{th:cex_lower_bound}
Let $0 < \multlow < 1 < \multupp$. Then, if $t \gg 1$, any interior point method whose trajectory is contained in the neighborhood $\Mcal_{\multlow, \multupp, t}(\R_{> 0})$ of the central path of~\eqref{eq:cex} requires at least $2^n - 1$ iterations to reduce the value of the objective function from $\Omega(1)$ to $1/t^{2u_n}$.  
\end{theorem}

\begin{proof}
We note that the supremum $\lambda^*$ of $x_n$ for $x$ satisfying~\eqref{eq:tcex} is equal to $0$. Thus, we can apply \Cref{th:complexity2} with $\ubar v = \Omega(1)$ and $\bar v = 1/t^{2u_n}$ to the sequence of segments corresponding to the trajectory followed by the interior point method. This proves that there is a least $\gamma([\log_t \ubar v, 2u_n])$ segments, and subsequently, iterations. We can easily see from \Cref{prop:trop_cp} that the latter quantity is equal to $\gamma([0, 2u_n])$, provided that $\log_t \ubar v < 1$, which can be assumed up to taking $t$ slightly larger. We conclude by \Cref{th:cex_gamma}.
\end{proof}

\begin{remark}
The exponential lower bound of \Cref{th:cex_lower_bound} is robust to changes on the initial and final values $\Omega(1)$ and $1/t^{2u_n}$ of the objective function. For instance, we have chosen the target value $1/t^{2 u_n}$ because this corresponds to the $1/2^{2L}$ threshold (where $L$ is the bitsize of the LP) used in linear programming after which the rounding method is applied to find an exact optimal solution. However, replacing $1/t^{2 u_n}$ by its square root $1/t^{u_n}$ only halves the lower bound (see \Cref{th:cex_gamma}).
\end{remark}

We conclude this section by some comments on the linear program~\eqref{eq:cex}.
We are guided by the structure of the tropical central path, which is reminiscent of that of the simplex path on the Klee--Minty cubes. More precisely, the first part of the tropical central path satisfies $x_{n-1} = 0$ and the other coordinates are negative. In terms of the linear program~\eqref{eq:cex}, this corresponds to the fact that the central path is close to the facet $\sum_{j = 1}^n x_j = 1$. In the last part of the tropical central path, we have $x_{n-1} = u_n + x_n$. This corresponds to a central path close to the facet $x_{n-1} = t^{u_n} x_n$.
\Cref{lemma:pair} shows that these faces are disjoint,
whereas~\Cref{lemma:pair2} shows that the inequalities in~\eqref{eq:cex} are paired, like in a cube.
In fact, we believe that the feasible set of~\eqref{eq:cex} is combinatorially equivalent to a $n$-cube. This is supported by the computation of the face lattice of the feasible set for $n \leq 5$. We leave the proof of this for a further work.

\section{Conclusion}

The combinatorics of the feasible set of~\eqref{eq:cex} is yet to be fully studied, in particular, the correspondence between the vertices and the $2^n$ extremities of the tropical segments in the tropical central path.

While we have used tropical geometry to construct complexity lower bounds to existing methods, we believe that it could be also a guide to develop new algorithms, that resist to such obstructions. In particular, it raises the question of identifying a trajectory to the optimal solution that cannot degenerate to the boundary of the tropical feasible set, \ie, that remains ``central'' in the tropical sense. 

\paragraph{Acknowledgments.} This work was partially done when the first author was visiting the Hausdorff Research Institute for Mathematics (HIM) during the ``Discrete Optimization'' trimester in Fall 2021. He thanks HIM for the support. He warmly thanks the organizers of the trimester. He specially thanks Daniel Dadush, Georg Loho, Bento Natura and L{\'a}szl{\'o} Végh for the stimulating discussions on this work. 

\bibliographystyle{alpha}

\appendix

\section{Proofs and auxiliary statements}
\label{app:proofs}

\begin{proof}[\proofname{} of~\Cref{cor:log_bound}]
Let $x,y \in D_f$. We define the function
\[
\phi(s) \coloneqq f(x + s(y-x)) \, .
\]
It is twice differentiable, and computing its derivatives yields $\phi'(s) = \inp{g(x+s(y-x))}{y-x}$ and $\phi''(s) = {(y-x)}^{\intercal} H(x+s(y-x)) (y-x)$. Using \cref{thm:hessian_ineq_lb}, we get
\[
\phi''(s) \geq \inp{(y-x)}{\frac{y-x}{n{\big(x+s(y-x)\big)}^2}} \, .
\]
Integrating this inequality gives $\phi'(s) \geq \inp{y-x}{\frac{1}{nx} - \frac{1}{n(x + s(y-x))}} + \phi'(0)$, and integrating once more yields
\[
\phi(1) - \phi(0) \geq \inp{y-x}{g(x) + \frac{1}{nx}} + \frac{1}{n} \sum_i \big( \log x_i - \log y_i \big) \, , 
\]
which is the desired result.
\end{proof}

\begin{lemma}
\label{lemma:component_ineq}
Let $x,y \in \R^n$. Suppose there are $a, b > 0$ such that $x \leq y + a$ and $\sum_i x_i + b \geq \sum_i y_i$.  Then $\distInf(x;y) \leq \max((n-1) a + b,a)\leq na+b$.
\end{lemma}

\begin{proof}[\proofname{} of \Cref{lemma:component_ineq}]
  Since $x\leq y +a$, we have $\max_i (x_i-y_i)\leq a$. 
  Moreover, $-b\leq \sum_i (x_i-y_i)\leq (n-1)\max_i (x_i-y_i)
  + \min_i (x_i-y_i)$, which entails
  that $\min_i(x_i-y_i)\geq -(n-1)a-b$.
  It follows that $\|x-y\|_\infty
  = \max(\max_i (x_i-y_i),-\min_i (x_i-y_i))
  \leq \max(a,(n-1)a+b)$.
\end{proof}

The proof of \cref{prop:non_archimedean_pointwise} relies on tropical properties of $\Ktrop$ to deduce metric properties on the family of convex sets $(\Kt)_t$. We start with a tropical perturbation lemma. We define first the metric $d_1$ on $\T^n$. The support of a point $u$ is the set $\supp(u) \coloneqq \{ i \colon u_i \neq -\infty \}$. Then $d_1(u,v)$ is set to be $+\infty$ if $\supp(u) \neq \supp(v)$, and otherwise
\[
d_1(u,v) \coloneqq \sum_{i \in \supp(u)} \abs{u_i - v_i}
\] 
We define $B_1(u;r)$ to be the open ball of center $u$ and radius $r$ for the $d_1$ distance.

For all $\eps > 0$, we define $\Ktrop^\eps$ the \textit{$\eps$-erosion} of $\Ktrop$ to be the set
\[
\Ktrop^\eps \coloneqq \{ u \in \Ktrop \colon B_1(u;\eps) \subset \Ktrop \} \, .
\]
We define \textit{$\linr(\Ktrop)$ the lower inner radius} of $\Ktrop$ to be the supremum of the radii for which the $\eps$-erosion $\Ktrop^\eps$ is not empty:
\[
\linr(\Ktrop) \coloneqq \sup\{ \eps \geq 0 \colon \Ktrop^\eps \neq \emptyset \} \, .
\]
Since $\Ktrop$ is supposed to be regular, then its interior is not empty, and $\linr(\Ktrop) > 0$.

\begin{lemma}
	\label{lemma:perturbation_lemma}
	There is $C_{\Ktrop} \geq 0$ depending only on $\Ktrop$ such that for all $\eps \in [0,\linr(\Ktrop)[$, we have $\distInf(\Ktrop, \Ktrop^\eps) \leq C_{\Ktrop} \eps$.
\end{lemma}

\begin{proof}
	Let us assume first that $\Ktrop$ is a polyhedra of the form $\Ktrop = \{ u \in \R \colon A u \leq b \}$ with $A \in \R^{m \times n}$ and $b \in \R^m$. Then 
	\[
	u \in \Ktrop^\eps \Leftrightarrow \forall j, u \pm \eps \e_j \in \Ktrop \Leftrightarrow \forall j, Au \leq b \pm \eps A_{\cdot j} \, .
	\]
	where $A_{\cdot j}$ denotes the $j^\textnormal{th}$ column of $A$. Therefore $\Ktrop^\eps = \{ u \colon Au \leq b_\eps \}$ where $b-b_\eps = \eps a$ with $a$ the positive vector defined by $a_i = \inf_j \abs{A_{ij}}$.
	
	Take then $u \in \Ktrop$. Consider the linear program
	\[
	\textnormal{minimize } \distInf(u, v) \textnormal{ subject to } v \in \Ktrop^\eps \, .
	\]
	Its dual can be written as
	\[
	\textnormal{maximize } w (Au - b_\eps) \textnormal{ subject to } w \geq 0 \textnormal{ and } \norm{\transpose{A}w}_1 \leq 1 \, .
	\]
	Since $\eps < \linr(\Ktrop)$, then the primal program is feasible. It is then also bounded. Therefore by strong duality, the dual is bounded as well and they have the same value. Its maximum is attained at one of the vertices of the polyhedron $\{w \geq 0,\norm{\transpose{A}w}_1 \leq 1\}$. We define $V$ to be this finite set of vectors with positive coordinates. Since $v$ also has positive coordinates, we have the following inequalities:
	\[
	\distInf(u, K^\eps) = \max_{w \in V} w(Au-b_\eps) \leq \max_{w \in V} w(b-b_\eps) = \eps \max_{w \in V} \inp{w}{a} \, .
	\]
	This yields the result by taking $C_{\Ktrop} \coloneqq \max_{w \in V} \inp{w}{a} \geq 0$.

Suppose now that $\Ktrop$ is only semilinear, meaning it is the finite union of polyhedra $\cup_i \Pcal_i$. Observe that for since $\Ktrop^\eps\subset \Ktrop$, $\distInf(\Ktrop,\Ktrop^\eps)=\sup_{x\in \Ktrop} \distInf(x,\Ktrop^\eps)$. By a compactness argument, the supremum is achieved, so 
        $\distInf(\Ktrop,\Ktrop^\eps)=\distInf(x^*,\Ktrop^\eps)$
        for some $x^*\in\Ktrop$. We have $x^*\in \Pcal_i$ from some $i$.
        Noting that $\Ktrop^\eps\supset \Pcal_i^\eps$, 
        and using the fact that the function $Y\mapsto \distInf(x^*,Y)$ is nonincreasing with respect to set inclusion, deduce that $\distInf(x^*,\Ktrop^\eps)
        \leq \distInf(x^*,\Pcal_i^\eps) \leq \distInf(\Pcal_i,\Pcal_i^\eps)
        \leq \eps C_{\Pcal_i}$.
        It follows that
        \[
	\distInf(\Ktrop, \Ktrop^\eps) \leq \eps C_{\Ktrop} \text{ with }C_{\Ktrop}:=\max_{i} C_{\Pcal_i} \enspace .
        \]
\end{proof}

\begin{proposition}
\label{prop:non_archimedean_pointwise}
Let $t > 1$ and $r > 0$. Let $x \in \Kt\cap (\R_{>0})^n$ such that $B_\infty(\log_t x; \frac{1}{2} \log_t n! + n \delta_\K(t) + r) \subset \Ktrop$. For all $y \in \R_{>0}^n$, if $\distInf(\log_t x;\log_t y) \leq r$, then $y \in \Kt$.
\end{proposition}

\begin{proof}

  It will be convenient to work with homogeneous coordinates.
  So, we associate to the convex set $\Kt\cap (\R_{>0})^n$
  the convex cone $\tilde{\K}_t\subset (\R_{>0})^{n+1}$ generated
  by the vectors of the form $\tilde{x}:=(1,x)$ with $x\in \Kt\cap (\R_{>0})^n$.
  Similarly, we associate to the tropical convex set $\Ktrop$
  the tropical convex cone $\tilde{\Ktrop}$ of $\R^{n+1}$, generated
  by the vectors of the form $(0,x)$ where $x\in \Ktrop\cap \R^n$
  (we use the same notation for the classical
  and the tropical homegeneisation, since no ambiguity will arise).
  The sup-norm distance $d_\infty$ on $\R^n$ induces a
  projective distance $\tilde{d}_\infty$ on $\R^{n+1}$. Here, projective
  is understood in the tropical sense,
  meaning that the distance is invariant under the additive action of scalars,
  i.e., $\tilde{d}_\infty( \lambda \unitAM  + u,v)=\tilde{d}_\infty(u,v)$
  holds for all $u,v\in \R^{n+1}$ and $\lambda\in\R$.
  Denoting by $\tilde{B}_\infty(\cdot;\cdot)$ the balls with
  respect to this projective distance, 
  the hypothesis then reads
  $\tilde{B}_\infty(\log_t \tilde{x}; \eps) \subset \tilde{\Ktrop}$,
  with $\eps \coloneqq \frac{1}{2} \log_t n! + n \delta_\K(t) + r$.

  Denote $u^i \coloneqq \log_t \tilde{x} + \eps \onevector_i$ for $i \in [n+1]$. Since by hypothesis on $\log_t \tilde{x}$ these points belong to
  $\tilde{\Ktrop}$, then we know that there exists $x^i \in \tilde{\K}_t$ such that for all $i \in [n+1]$, $\distInf(\log_t x^i;u^i) \leq \delta_\K(t)$. In other words, the $x^i$ are such that for all $i,j \in [n+1]$,
	\[
	t^{-\delta_\K(t) + \delta_{ij} \eps} \tilde{x}_j \leq x^i_j \leq t^{\delta_\K(t) + \delta_{ij} \eps} \tilde{x}_j \, .
	\]
	We claim that for all $y \in \R^n$ such that $\distInf(\log_t y;\log_t x) \leq r$, the point $\tilde{y} = (1,y) \in \R^{n+1}$ belongs to the cone generated by the vectors $x^i$. This cone is defined by the following inequalities for $i \in [n+1]$:
	\[
	\det(x^1, \dots, x^{n+1}) \det(x^1, \dots, x^{i-1}, \tilde{y}, x^{i+1}, \dots, x^{n+1})	\geq 0 \, .
	\]
	We expand the first determinants as follows:
	\begin{equation}
		\label{eq:non_archimidean_pointwise_1}
		\det(x^1, \dots, x^{n+1}) = \sum_{\sigma \in S_{n+1}} \textnormal{sign}(\sigma) \prod_{j=1}^{n+1} x^j_{\sigma(j)} \, .
	\end{equation}
	Looking at the terms of the sum, since a permutation $\sigma \neq \textnormal{Id}$ has at most $(n-2)$ fixed points, we get
	\[
	\prod_{j=1}^{n+1} x^j_{\sigma(j)} \leq t^{n\delta_\K(t) + (n-2) \eps} \prod_{j=1}^{n+1} \tilde{x}_j  \, .
	\]
	For the term corresponding to $\sigma = \textnormal{Id}$, we have
	\[
	\prod_{j=1}^{n+1} x^j_j \geq t^{-n \delta_\K(t) + n \eps} \prod_{j=1}^{n+1} \tilde{x}_j \, .
	\]
	Since $\frac{1}{2} \log_t n! + n \delta_\K(t) \leq \eps$, then $n! t^{n \delta_\K(t) + (n-2) \eps} \leq t^{- n \delta_\K(t)+n\eps}$. This yields 
	\[
	\abs{\sum_{\sigma \neq \textnormal{Id}} \textnormal{sign}(\sigma) \prod_{j=1}^{n+1} x^j_{\sigma(j)}} \leq \sum_{\sigma \neq \textnormal{Id}} \prod_{j=1}^{n+1} x^j_{\sigma(j)} \leq n! t^{n \delta_\K(t) + (n-2) \eps} \prod_{j=1}^{n+1} \tilde{x}_j \leq t^{-n \delta_\K(t) + n \eps} \prod_{j=1}^{n+1} \tilde{x}_j \leq \prod_{j=1}^{n+1} x^j_j \, .
	\]
	As a result, in the sum defined in \cref{eq:non_archimidean_pointwise_1}, the term associated to the identity dominates the terms associated to the other permutations and is positive. The determinant is therefore positive.
	
	Likewise, for $i \in [n+1]$, developing the determinant
	\begin{equation}
		\label{eq:non_archimidean_pointwise_2}
		\det(x^1, \dots, x^{i-1}, \tilde{y}, x^{i+1}, \dots, x^{n+1}) = \sum_{\sigma \in S_{n+1}} \textnormal{sign}(\sigma) \tilde{y}_{\sigma(i)} \prod_{\substack{j = 1 \\ j \neq i}}^{n+1} x^j_{\sigma(j)} \, ,
	\end{equation}
	and using that $\distInf(\log_t y;\log_t x) \leq r$, we see that for $\sigma \neq \textnormal{Id}$ a permutation, we have
	\[
	\tilde{y}_{\sigma(i)} \prod_{\substack{j = 1 \\ j \neq i}}^{n+1} x^j_{\sigma(j)} \leq t^{r + (n-1)\delta_\K(t) + (n-2) \eps} \prod_{j=1}^{n+1} \tilde{x}_j
	\]
	and
	\[
	\tilde{y}_i \prod_{\substack{j = 1 \\ j \neq i}}^{n+1} \tilde{x}^j_i \geq t^{-r-(n-1) \delta_\K(t) + n \eps}  \prod_{j=1}^{n+1} \tilde{x}_j \, .
	\]
	Since $\frac{1}{2}n! + (n-1) \delta_\K(t) + r \leq \eps$, then the determinant is also positive.
	
	As a result, $\tilde{y} \in \textnormal{cone}({x^0, \dots, x^n}) \subset \tilde{\K}_t$, meaning that $y \in \Kt$.
\end{proof}

\begin{proof}[\proofname~(\cref{lemma:good_lift})]
	Consider $t_K > 1$ such that for $t \geq t_K$, $\eps(t) \coloneqq \delta_c(t) + (n+1) \delta_\K(t) + \log_t (K+2) + \frac{1}{2} \log_t n! + \log_t n \leq \linr(\Ktrop)$. Denote $\lambda = - \log_t \eta$, and $b^\lambda = \tbar \Ktrop_\lambda$. Following \cref{lemma:perturbation_lemma}, we know that there is $u \in \Ktrop$ such that $B_\infty(u;\eps(t)) \subset \Ktrop_\lambda$ and $\distInf(u; b^\lambda(t)) \leq C_{\Ktrop} \eps(t)$. 
	
	Consider $x \in \Kt$ such that $\distInf(\log_t x; u) \leq \delta_\K(t)$. Then using the relation~(\ref{eq:log_ineq}) and the linearity of the inner product, we have
	\[
	\log_t \inp{c_t}{x} \leq \log_t n + \delta_c(t) + \delta_\K(t) + \inpT{c}{u} \leq \log_t n + \delta_c(t) + \delta_\K(t) + \lambda - \eps(t) \leq \lambda \, ,
	\]
	meaning that $\inp{c_t}{x} \leq \frac{1}{\eta}$.
	
	Furthermore, for any $y \in \Kt$ such that $\inp{c_t}{y} \leq \frac{1}{\eta}$, using once more the relation~(\ref{eq:log_ineq}), we get
	\[
	\inpT{c}{\log_t y} \leq \delta_c(t) + \inpT{\log_t c_t}{\log_t y} \leq \delta_c(t) + \log_t \inp{c_t}{y} \leq \delta_c(t) + \lambda \, .
	\]
	Therefore $\log_t y \leq b^\lambda + \delta_c(t) \leq \log_t x + C_{\Ktrop} \eps(t) + \delta_\K(t) + \delta_c(t)$, meaning
	\[
	y \leq x t^{C_{\Ktrop} \eps(t)} t^{\delta_\K(t) + \delta_c(t)} \leq \alpha x  \, .
	\]
	where $\alpha \coloneqq \exp(C_{\Ktrop}\big((n+2)\delta^* + \log(K+2) + \frac{1}{2} \log(n!) + \log(n)\big) + 2\delta^*)$.
	
	Finally, denoting $a^i \coloneqq x + (K+1) x_i \onevector_i$, we have $\distInf(\log_t a^i; \log_t x) \leq \log_t(K+2)$. Since $B_\infty(\log_t x; \eps(t) - \delta_\K(t)) \subset \Ktrop$ and $\eps(t) - \delta(t) \geq \frac{1}{2} n! + n \delta(t) + \log_t(C+3)$, then using \cref{prop:non_archimedean_pointwise}, we have $a^i \in \Kt$. Likewise, denoting $b^i \coloneqq x - \frac{1}{2} x_i$, we have $\distInf(\log_t b_i; \log_t x) \leq \log_t 2 \leqslant \log_t(K+2)$, therefore $b^i \in \Kt$.
\end{proof}

\section{Interior point methods}\label{app:IPM}

We consider a convex optimization problem of the form~\eqref{eq:convex}, where $\Kcal$ is a closed convex set, and suppose that its optimal value, which we denote $\optval$, is finite. Given a self-concordant barrier $f$ over $\interior \Kcal$, we denote by $\cp(\eta)$ the point of the central path of parameter $\eta \geq 0$, \ie, the unique optimal solution of~\eqref{eq:penalized}. 
The following lemma justifies the interest of the central path by providing a bound of the value function along it. It shows that as $\eta \to +\infty$, the point of the central path $\cp(\eta)$ converges to a solution of~\eqref{eq:convex}. 
\begin{lemma}
\label{lemma:cp_val_ineq}
For all $\eta>0$, we have $\inp{c}{\cp(\eta)} \leq \optval + \frac{\thet}{\eta}$.
\end{lemma}

\begin{proof}
We have $g(\cp(\eta)) = -\eta c$ by the optimality of $\cp(\eta)$. By~\cite[Theorem 2.3.3]{Renegar01}, we know that for all $x, y \in D_f$, $\inp{g(x)}{y-x} \leq \thet$. By applying this to $\cp(\eta)$ and an optimal solution $x^*$ of~\eqref{eq:convex}, we obtain:
\[
\inp{c}{\cp(\eta)} - \inp{c}{x^*} = \frac{1}{\eta} \inp{g(\cp(\eta))}{x^* - \cp(\eta)} \leq \frac{\thet}{\eta} \, .
\]
Since $\optval = \inp{c}{x^*}$, the result follows.
\end{proof}

\section{Analysis of Example~\eqref{eq:tcex}}\label{app:cube}

We collect here results and proofs needed in the analysis of the example of linear program~\eqref{eq:tcex}.

\begin{lemma}\label{lemma:regular}
The set $\Ptrop$ of points $x \in \T^n$ that satisfy the constraints of~\eqref{eq:tcex} is a regular set.	
\end{lemma}

\begin{proof}
First, suppose that $x \in \Ptrop\cap \R^n$. Then, it can be verified that the point $x - \eps \begin{bsmallmatrix}
2^{n-1} &
2^{n-2} &
\dots &
1 	
\end{bsmallmatrix}^\top$ satisfies the inequalities of~\eqref{eq:tcex} in a strict way, for all $\eps > 0$. We deduce that this point lies in the interior of $\Ptrop$. Moreover, taking $\eps \to 0^+$ shows that $x$ is in the closure of the interior of $\Ptrop$.

Suppose then that $x\in \Ptrop$ has $-\infty$ entries. Then the third constraint of~\eqref{eq:tcex} implies that the set $\{ i \colon x_i > -\infty \}$ is of the form $I_k \coloneqq \{ k+1, \dots, n \}$ for some $k \in [n]$. Denote then $x'(\eps,\nu)$ the point defined by $x'_j(\eps, \nu) = x_j - \eps 2^{n-j}$ for $j \in I_k$ and $x'_j(\eps,\nu) = -\nu 2^{n-j}$ for $j \notin I_k$. We first set $\epsilon_0 > 0$. Suppose there exists $\nu_0>0$ such that $x'(\eps_0,\nu_0)$ satisfies the constraints~\eqref{eq:tcex}. Then all the $x'(\eps,\nu)$ also satisfy the constraint for $\nu > \nu_0$ and $\eps \in ]0,\epsilon_0[$. These points then lie in the interior of $\Ptrop$, and taking $\epsilon \to 0^+$ and $\lambda \to +\infty$ shows that $x$ is in the closure $\interior \Ptrop$. We show therefore that any $\nu_0$ that is big enough will be such that $x'(\eps_0,\nu_0)$ satisfies the constraints.

The coefficients of $x'(\eps_0,\nu_0)$ are such that all of the inequalities of the third constraint are satisfied for any $\nu_0$ except the inequality $x_k < x_{k+1}$. However rewriting it as $-\nu_0 \leq x_{k+1} - \eps_0$, we see that it is satisfied for $\nu_0$ large enough. The second constraint is obviously verified. For the first constraint to hold, notice that we can take $\nu_0$ large enough to satisfy for all $i \in [k-1]$
\[
	-u_i - 2^{n-i} \nu_0 < -u_{i+1} - 2^{n-i-1} \nu_0 \, ,
\]
in which case the first $k$ inequalities of the constraint hold true. Furthermore, since the inequality holds for $x$, we can choose $\nu_0$ large enough for the following inequalities to hold true as well for $i = k+1, \dots, n$:
\begin{multline*}
	\bigvee_{j = 1}^{k} (-u_i - 2^{n-j} \nu_0) \bigvee_{j = k+1}^{i-1} (-u_j + x_j - 2^{n-j} \eps_0)\vee (-u_{i+1} + 1 - 2^{n-i}\eps_0)  \\
	< \bigvee_{j = i+1}^{n-1} (-u_j + x_j - 2^{n-j} \eps_0) \vee (x_n - \eps_0) \vee (-u_n) \, ,
\end{multline*}
which concludes the proof.
\end{proof}
\begin{lemma}\label{lemma:isval}
  The parametric family of linear programs~\eqref{eq:cex} satisfies~\Cref{assmpt:assmpt_K_c}, and the log-limit of the feasible sets $\Pt$ of~\eqref{eq:cex} coincide with the feasible set $\Ptrop$ of the tropical linear program~\eqref{eq:tcex}.
\end{lemma}
\begin{proof}
  Following~\cite{ABGJ18}, we identify the family $\Pt$ with a single polyhedron $\bm\Ptrop$ over an ordered field $\mathbb{F}$ of Puiseux series, equipped
  with the nonarchimedean valuation $\val$ which coincides with the log-limit.
  We claim that
  \begin{align}
\Ptrop \cap \R^n = \val(\bm \Ptrop \cap (\mathbb{F}_{>0})^n) \subset \val(\bm\Ptrop) \subset \Ptrop\enspace .\label{e:sandwitch}
  \end{align}
  Indeed, the first equality follows from~\cite[Coro.~4.8]{tropicalspectrahedra}, since by \cref{lemma:regular}, we know that $\Ptrop \cap \R^n$ is regular.
  The first inclusion is trivial. The second inclusion follows from the fact that the nonarchimedean valuation is a morphism from the semifield $(\mathbb{F}_{\geq 0}, +,\times)$ to the tropical semifield $(\T,\vee,+)$ (cf.~\eqref{e:morphism}),
  so that the valuation
  of any element of $\bm\Ptrop$ satisfies the inequalities obtained by ``tropicalizing'' the defining inequalities of $\Ptrop$.
 Further,
  by \cite[Theorem 4.1]{tropicalspectrahedra}, since $\bm \Ptrop$ is closed, then its image $\val(\bm \Ptrop)$ is closed as well in $\T^n$.
  Using again~\Cref{lemma:regular}, we get that $\Ptrop$ is the closure of $\Ptrop\cap \R^n$ in $\T^n$. Hence, applying the closure mapping $\operatorname{clo}(\cdot)$ to the chain of inclusions~\eqref{e:sandwitch}, we get
  \begin{align*}
    \Ptrop=\operatorname{clo}(\Ptrop \cap \R^n)
 \subset \val(\bm\Ptrop) \subset \operatorname{clo}(\Ptrop)=\Ptrop\enspace .
  \end{align*}
  So,  $\Ptrop=\val(\bm\Ptrop)$ and by
  \cite[Theorem~12]{ABGJ18},
  $\distInf(\log_t \Pt , \Ptrop) = O(1/(\log t))$, 
  showing that the second statement of the present lemma
  as well as the first condition of~\Cref{assmpt:assmpt_K_c}
  hold.   The second of these conditions follows from~\eqref{lemma:regular},
  whereas the third one follows from~\cite[Th.~3.1]{tropicalspectrahedra}.
\end{proof}
\begin{proof}[Proof of~\Cref{prop:trop_cp}]
The description of $\cptrop_1$ follows from Lemma~\ref{lemma:x_n}, as well the fact that $\cptrop_n(\nu) = 0$ for $\lambda \leq 0$.

We now suppose $n > 1$, and consider $\lambda \geq 0$. We set $z \coloneqq \cptrop_n(\lambda)$ for short, and we denote by $z'$ the right-hand side of the equation giving $\cptrop_n(\lambda)$ in \Cref{prop:trop_cp}. 

It is useful to restate the inequalities of~\eqref{eq:tcex} as follows:
\begin{align}
\allowdisplaybreaks
\bigvee_{j = 1}^{i-1} (-u_i + x_j) \vee (-u_{i+1} + 1 + x_i) & \leq \bigvee_{j = i+1}^{n-1} (-u_j + x_j) \vee x_n \vee (-u_n) && \text{for all}\enspace i = 1, \dots, n-2 \label{eq:eq1} \\ 
\bigvee_{j = 1}^{n-2} (-u_{n-1} + x_j) \vee (-u_n + 1 + x_{n-1}) & \leq x_n \vee (-u_n) \label{eq:eq2} \\ 
\bigvee_{j = 1}^n x_j & \leq 0 \label{eq:eq3} \\ 
-\infty \leq x_1 \leq x_2 & \leq \dots \leq x_{n-1} \leq u_n + x_n \label{eq:eq4} \, ,
\end{align}
and to recall that of~$\tcex_{n-1}$:
\begin{align}
\allowdisplaybreaks
\bigvee_{j = 1}^{i-1} (-u_i + x_j) \vee (-u_{i+1} + 1 + x_i) & \leq \bigvee_{j = i+1}^{n-2} (-u_j + x_j) \vee x_{n-1} \vee (-u_{n-1}) && \text{for all}\enspace i = 1, \dots, n-2 \label{eq:eq5} \\ 
\bigvee_{j = 1}^{n-1} x_j & \leq 0 \label{eq:eq6} \\ 
-\infty \leq x_1 \leq x_2 & \leq \dots \leq x_{n-2} \leq u_{n-1} + x_{n-1} \label{eq:eq7} \, .
\end{align}

We note that $z' \leq 0$, and $z'_n = -\lambda$. As a consequence, in order to show that $z = z'$, it suffices to show that $z \leq z'$, and that $z'$ satisfies~\eqref{eq:eq1}, \eqref{eq:eq2} and~\eqref{eq:eq4}. We distinguish three cases, according to the value of $\lambda \geq 0$.

\begin{description}
\item[$0 \leq \lambda \leq u_n-1$.] 
We claim that the point $\begin{bsmallmatrix}z_{[n-2]} \\ z_n\end{bsmallmatrix}$ satisfies the constraints of $\tcex_{n-1}$. Indeed, we have $z_{n-1} \leq 0$ by~\eqref{eq:eq3} and $u_{n-1} \leq u_n$, so that $(-u_{n-1} + z_{n-1}) \vee (-u_n) \leq -u_{n-1}$. We deduce from~\eqref{eq:eq1} that:
\[
\bigvee_{j = 1}^{i-1} (-u_i + z_j) \vee (-u_{i+1} + 1 + z_i) \leq \bigvee_{j = i+1}^{n-2} (-u_j + z_j) \vee x_n \vee (-u_{n-1}) \qquad \text{for all}\; i = 1, \dots, n-2 \, .
\]
Moreover, $\bigvee_{j = 1}^{n-2} z_j \vee z_n \leq \bigvee_{j = 1}^n z_j \leq 0$ (from~\eqref{eq:eq3}), and $-\infty \leq z_1 \leq z_2 \leq \dots \leq z_{n-2}$ (from~\eqref{eq:eq4}). Finally, from~\eqref{eq:eq2}, we know that $-u_{n-1} + z_{n-2} \leq z_n \vee (-u_n) = z_n$, where the last equality comes from $z_n = -\lambda$ (Lemma~\ref{lemma:x_n}) and $-\lambda \geq -u_n + 1$. 

Since $z_n = -\lambda$, we deduce that $(z_{[n-2]}, z_n) \leq \cptrop_{n-1}(\lambda)$. As $z_{n-1} \leq 0$, this ensures that $z \leq z'$. 

We know check that $z'$ verifies~\eqref{eq:eq1}, \eqref{eq:eq2} and~\eqref{eq:eq4}. Recall that $z'_{n-1} = 0$. We remark that the constraints~\eqref{eq:eq1} are satisfied because $(-u_{n-1} + z'_{n-1}) \vee (-u_n) = -u_{n-1}$ and $\cptrop_{n-1}(\lambda) = \begin{bsmallmatrix} z'_1 \\ \vdots \\ z'_{n-2} \\ -\lambda \end{bsmallmatrix}$ satisfies~\eqref{eq:eq5}. Moreover, we have $-u_{n-1} + z'_j \leq z'_n$ for all $j \leq n-2$ because $z'_1 \leq \dots \leq z'_{n-2} \leq u_{n-1} + z'_n$ (thanks to~\eqref{eq:eq7}). Since $-u_n + 1 + z'_{n-1} = -u_n + 1 \leq \lambda = z'_n$, we deduce that the constraint~\eqref{eq:eq2} is satisfied by $z'$. Finally, \eqref{eq:eq4} is satisfied thanks to the fact that $z'_1 \leq z'_2 \leq \dots \leq z'_{n-2} \leq 0 = z'_{n-1}$, and $z'_{n-1} = 0 \leq u_n - \lambda = u_n + z'_n$. This completes the proof that $z = z'$.

\item[$u_n-1 \leq \lambda \leq u_n + 1$.] By~\eqref{eq:eq2}, we have, for all $j < n-1$, 
\begin{align*}
\allowdisplaybreaks
z_j & \leq u_{n-1} + (z_n \vee (-u_n)) \\
& \leq -u_{n-1} + ((z_n + 2u_{n-1}) \vee (2u_{n-1}-u_n)) \\
& = -u_{n-1} + ((u_n-1) -\lambda) \vee (-1)) && \text{as $2 u_{n-1} = u_n - 1$ and $z_n \leq -\lambda$}\\
& = z'_j \, .
\end{align*}
Similarly, by~\eqref{eq:eq2}, $z_{n-1} \leq u_n - 1 + (z_n \vee (-u_n)) \leq ((u_n-1) -\lambda) \vee (-1))$. We deduce that $z \leq z'$.

It now remains to show that $z'$ satisfies~\eqref{eq:eq1}, \eqref{eq:eq2} and~\eqref{eq:eq4}. We note that, for all $j < n-1$, we have $z'_j = -u_{n-1} + z'_{n-1}$. If $1 \leq i < n-2$. Since $-u_i \leq 0$ and $-u_{i+1} + 1 \leq 0$, we have $\bigvee_{j = 1}^{i-1} (-u_i + x_j) \vee (-u_{i+1} + 1 + x_i) \leq -u_{n-1} + z'_{n-1}$. We deduce that~\eqref{eq:eq1} is satisfied by $z'$.
Similarly, if $j < n-1$, we have 
\begin{align*}
\allowdisplaybreaks
-u_{n-1} + z'_j & = -2u_{n-1} + ((u_n - 1) - \lambda) \vee (-1)) \\
& = -(u_n - 1) + ((u_n - 1) - \lambda) \vee (-1)) \\
& \leq (-\lambda) \vee (-u_n) = z'_n \vee (-u_n) \, .
\end{align*}
Similarly, $-u_n + 1 + z'_{n-1} = (-\lambda) \vee (-u_n) = z'_n \vee (-u_n)$. Thus, \eqref{eq:eq2} is satisfied. The inequalities $z'_1 \leq \dots \leq z'_{n-1}$ are immediate from $u_{n-1} \geq 0$. Finally, 
\begin{align*}
\allowdisplaybreaks
z'_{n-1} & = (u_n - 1) - \lambda) \vee (-1) \\
& = u_n + ((-\lambda - 1) \vee (-u_n - 1))  \\
& = u_n + ((-\lambda -1) \vee (-\lambda)) && \text{since $\lambda \leq u_n+ 1$} \\
& \leq u_n - \lambda = u_n + z'_n \, .
\end{align*}
This shows that~\eqref{eq:eq4} is satisfied, and, in turn, that $z = z'$.

\item[$u_n+1 \leq \lambda$.] We note that $z_n = -\lambda \leq -u_n$. As a consequence of~\eqref{eq:eq1}, $z$ satisfies the following inequalities:
\begin{align}
\bigvee_{j = 1}^{i-1} (-u_i + z_j) \vee (-u_{i+1} + 1 + z_i) & \leq \bigvee_{j = i+1}^{n-1} (-u_j + z_j) \vee (-u_n) && \text{for all}\enspace i = 1, \dots, n-2 \, .
\end{align}
Setting $\bar z \coloneqq z + \begin{bsmallmatrix} 
(u_{n-1} + 1)e \\
1 \\
(u_n + 1) 
\end{bsmallmatrix}$ and adding $u_{n-1} + 1$ to both sides of the inequalities above precisely shows that $\bar z_{[n-1]}$ satisfies~\eqref{eq:eq5} (recall that $-u_n + (u_{n-1} + 1) = -u_{n-1}$). Besides, \eqref{eq:eq2} ensures that for all $j < n-1$, $-u_{n-1} + z_j \leq z_n \vee -u_n$, which amounts to $z_j + u_{n-1} + 1 \leq 0$ since $z_n \leq -u_n$. Similarly, $z_{n-1} + 1 \leq 0$. Thus, $\bar z_{[n-1]}$ satisfies~\eqref{eq:eq6}. By \eqref{eq:eq4}, we trivially have $\bar z_1 \leq \dots \leq \bar z_{n-2}$, and $\bar z_{n-2} = z_{n-2} + u_{n-1} + 1 \leq z_{n-1} + u_{n-1} + 1 = u_{n-1} + \bar z_{n-1}$. Thus, $\bar z_{[n-1]}$ satisfies~\eqref{eq:eq7}. As a consequence, $\bar z_{[n-1]}$ satisfies the inequalities of~($\tcex_{n-1}$). Since $\bar z_{n-1} = z_{n-1} + 1 \leq u_n + z_n + 1 = (u_n + 1) - \lambda$, we deduce that $\bar z_{[n-1]} \leq \cptrop_{n-1}(\lambda - (u_n + 1))$, or, equivalently, $z_{[n-1]} \leq \begin{bsmallmatrix} (u_{n-1} + 1) e \\ 1 \end{bsmallmatrix} + \cptrop_{n-1}(\lambda - (u_n + 1)) = z'_{[n-1]}$. As $z_n = -\lambda = z'_n$, this proves that $z \leq z'$.

We now check that $z'$ verifies the constraints~\eqref{eq:eq1}, \eqref{eq:eq2} and~\eqref{eq:eq4}. Recall that $\cptrop_{n-1}(\lambda - (u_n + 1))$ satisfies~\eqref{eq:eq5}. Adding $-(u_{n-1} + 1)$ to both sides of these inequalities and exploiting the fact that $-u_{n-1} - (u_{n-1} + 1) = -u_n$ shows that
\begin{align}
\bigvee_{j = 1}^{i-1} (-u_i + z'_j) \vee (-u_{i+1} + 1 + z'_i) & \leq \bigvee_{j = i+1}^{n-1} (-u_j + z'_j) \vee (-u_n) && \text{for all}\enspace i = 1, \dots, n-2 \, .
\end{align}
We deduce that~\eqref{eq:eq1} are satisfied by $z'$. Moreover, for all $j < n-1$, $-u_{n-1} + z'_j = -u_n + \big(\cptrop_{n-1}(\lambda - (u_n + 1))\big)_j \leq -u_n$. Besides, $z'_{n-1} = -1 + (u_n + 1 - \lambda)$ thanks to Lemma~\eqref{lemma:x_n}, and so $-u_n + 1 + z'_{n-1} = 1 -\lambda \leq -u_n$. As $-u_n \leq z'_n \vee (-u_n)$, this shows that $z'$ satisfies~\eqref{eq:eq6}. Finally, the fact that $z'_1 \leq \dots \leq z'_{n-1}$ follows from the fact that $\cptrop_{n-1}(\lambda - (u_n + 1))$ satisfies~\eqref{eq:eq4}. Finally, $z'_{n-1} = -1 + (u_n + 1 - \lambda) = u_n - z'_n$. We deduce that $z = z'$.\qedhere 
\end{description}
\end{proof}

In the following lemma, we show that the two $n-1$ dimensional faces successively visited by the central path are disjoint:
\begin{lemma}\label{lemma:pair}
If $t$ is sufficiently large, the faces defined by the equalities $\sum_{j = 1}^n x_j \leq 1$ and $x_{n-1} \leq t^{u_n} x_n$ are disjoint.
\end{lemma}

\begin{proof}
Let $x$ be a feasible point, and suppose that $x_{n-1} = t^{u_n} x_n$. Taking $i = n-1$ in the first set of inequalities shows that
\[
t^{-u_{n-1}} \Big(\sum_{j = 1}^{n-2} x_j\Big) + t^{-u_n + 1} x_{n-1} \leq x_n + t^{-u_n} \, .
\]	
Since the $x_j$ are nonnegative, we deduce that 
\[
t x_{n-1} \leq t^{u_n} x_n + 1 = x_{n-1} + 1\, .
\]
Therefore, $x_{n-1} \leq \frac{1}{t-1}$. As $x_j \leq x_{n-1}$ and $x_n = t^{-u_n} x_{n-1} \leq x_{n-1}$, we cannot have $\sum_{j = 1}^n x_j \leq 1$ when $t \gg 1$.
\end{proof}

The following lemma actually shows that the inequalities in~\eqref{eq:cex} are paired, like in a cube:
\begin{lemma}\label{lemma:pair2}
Let $i < n$. If $t$ is sufficiently large, the faces defined by the equalities $x_i = x_{i+1}$ and 
\begin{equation}\label{eq:face}
t^{-u_{i+1}} \Big(\sum_{j \leq i} x_j\Big) + t^{-u_{i+2} + 1} x_{i+1} = \sum_{j = i+2}^{n-1} t^{-u_j} x_j + x_n + t^{-u_n} \, .
\end{equation}
are disjoint.
\end{lemma}

\begin{proof}
Let $x$ be a feasible point such that $x_i = x_{i+1}$. The $i$th inequality in the first set of inequalities defining~\eqref{eq:cex} and the fact that $x_j \geq 0$ for $j < i$ entails:
\begin{equation}\label{eq:pair1}
(t^{-u_{i+1} + 1} - t^{-u_{i+1}}) x_i \leq \sum_{j = i+2}^{n-1} t^{-u_j} x_j + x_n + t^{-u_n} \; .
\end{equation}
Suppose that~\eqref{eq:face} holds. Since $x_j \leq x_i$ for all $j \leq i$, we have 
\begin{equation}\label{eq:pair2}
(i t^{-u_{i+1}} + t^{-u_{i+2} + 1}) x_i \geq \sum_{j = i+2}^{n-1} t^{-u_j} x_j + x_n + t^{-u_n} \, .
\end{equation}
We deduce from~\eqref{eq:pair1} and~\eqref{eq:pair2} that
\[
(t-1) x_i \leq (i + t^{-u_{i+1}}) x_i \, .
\]
Note that $x_i > 0$ (because $(i t^{-u_{i+1}} + t^{-u_{i+2} + 1}) x_i \geq t^{-u_n}$ from~\eqref{eq:pair2}). Thus, we reduce to the inequality $t-1 \leq i + t^{-u_{i+1}}$, which cannot hold when $t \gg 1$.
\end{proof}

\end{document}